\documentclass[12pt]{amsart}

\usepackage{amssymb,amsthm,amsmath,amstext,amsxtra}
\usepackage{fullpage}
\usepackage{bm}       
\usepackage{mathtools} 
\usepackage[all]{xy}
\usepackage{booktabs}
\usepackage{hyperref}
\hypersetup{colorlinks=true,urlcolor=blue,citecolor=blue,linkcolor=blue}
\usepackage{enumerate}
\usepackage{ stmaryrd }
\usepackage{enumitem}
\usepackage{comment}
\usepackage{colonequals}
\usepackage{tikz}
\usepackage{float} 
\usepackage{subcaption}
\usepackage{breqn}

\usepackage{eqparbox}
\usepackage[numbers]{natbib}

\makeatletter
\newcommand*{\da@rightarrow}{\mathchar"0\hexnumber@\symAMSa 4B }
\newcommand*{\da@leftarrow}{\mathchar"0\hexnumber@\symAMSa 4C }
\newcommand*{\xdashrightarrow}[2][]{%
  \mathrel{%
    \mathpalette{\da@xarrow{#1}{#2}{}\da@rightarrow{\,}{}}{}%
  }%
}
\newcommand{\xdashleftarrow}[2][]{%
  \mathrel{%
    \mathpalette{\da@xarrow{#1}{#2}\da@leftarrow{}{}{\,}}{}%
  }%
}
\newcommand*{\da@xarrow}[7]{%
  \sbox0{$\ifx#7\scriptstyle\scriptscriptstyle\else\scriptstyle\fi#5#1#6\m@th$}%
  \sbox2{$\ifx#7\scriptstyle\scriptscriptstyle\else\scriptstyle\fi#5#2#6\m@th$}%
  \sbox4{$#7\dabar@\m@th$}%
  \dimen@=\wd0 %
  \ifdim\wd2 >\dimen@
    \dimen@=\wd2 %
  \fi
  \count@=2 %
  \def\da@bars{\dabar@\dabar@}%
  \@whiledim\count@\wd4<\dimen@\do{%
    \advance\count@\@ne
    \expandafter\def\expandafter\da@bars\expandafter{%
      \da@bars
      \dabar@ 
    }%
  }%
  \mathrel{#3}%
  \mathrel{%
    \mathop{\da@bars}\limits
    \ifx\\#1\\%
    \else
      _{\copy0}%
    \fi
    \ifx\\#2\\%
    \else
      ^{\copy2}%
    \fi
  }%
  \mathrel{#4}%
}
\makeatother

\newcommand{\QMtwosix}{\texttt{262144.d.524288.1}}

\numberwithin{equation}{subsection}

\theoremstyle{plain}
\newtheorem*{theorema*}{Theorem A}

\newtheorem{proposition}[equation]{Proposition}
\newtheorem{lemma}[equation]{Lemma}
\newtheorem{corollary}[equation]{Corollary}

\newtheorem{hypothesis}[equation]{Hypothesis}

\theoremstyle{definition}

\newtheorem{example}[equation]{Example}

\theoremstyle{remark}
\newtheorem{remark}[equation]{Remark}

\newenvironment{enumalph}
{\begin{enumerate}}
{\end{enumerate}}

\newenvironment{enumroman}
{\begin{enumerate}}
{\end{enumerate}}

\newcommand{\Z}{\mathbb{Z}}
\newcommand{\Q}{\mathbb{Q}}
\newcommand{\R}{\mathbb{R}}
\newcommand{\F}{\mathbb{F}}

\let\C\relax
\newcommand{\C}{\mathbb{C}}
\let\P\relax
\newcommand{\P}{\mathbb{P}}
\newcommand{\PP}{\mathbb{P}}

\newcommand{\Fp}{{{\F}_{\frakp}}}

\renewcommand{\d}{\operatorname{d} \!}
\newcommand{\Infty}{\infty_X}

\newcommand{\defi}[1]{\textsf{#1}} 	

\let\Re\relax
\DeclareMathOperator{\Re}{Re}
\let\Im\relax
\DeclareMathOperator{\Im}{Im}

\DeclareMathOperator{\Jac}{Jac}

\DeclareMathOperator{\M}{M}

\DeclareMathOperator{\opchar}{char}

\DeclareMathOperator{\diag}{diag}
\DeclareMathOperator{\Nm}{Nm}

\let\U\relax
\DeclareMathOperator{\U}{U}

\DeclareMathOperator{\SU}{SU}
\DeclareMathOperator{\Gal}{Gal}

\DeclareMathOperator{\Frob}{Frob}

\DeclareMathOperator{\AJ}{AJ}
\DeclareMathOperator{\Mum}{Mum}

\newcommand{\et}{\mathrm{\acute{e}t}}

\DeclareMathOperator{\opdiv}{div}
\DeclareMathOperator{\lcm}{\mathrm{lcm}}
\DeclareMathOperator{\Div}{Div}
\DeclareMathOperator{\tr}{\mathrm{tr}}

\DeclareMathOperator{\disc}{\mathrm{disc}}
\DeclareMathOperator{\rk}{\mathrm{rk}}

\DeclareMathOperator{\End}{\mathrm{End}}

\DeclareMathOperator{\NS}{\mathrm{NS}}
\DeclareMathOperator{\Sym}{\mathrm{Sym}}
\DeclareMathOperator{\Pic}{\mathrm{Pic}}
\DeclareMathOperator{\Br}{\mathrm{Br}}

\DeclareMathOperator{\alg}{{al}}

\newcommand{\frakp}{\mathfrak{p}}

\newcommand{\Qbar}{{\overline{\Q}}}

\newcommand{\Magma}{{\textsc{Magma}}}

\newcommand{\SageMath}{{\textsc{SageMath}}}

\newcommand{\psmod}[1]{~(\textup{\text{mod}}~{#1})}

\newcommand{\FAconn}{F_A^{\textup{\tiny{conn}}}}

\usepackage{xcolor}
\definecolor{darkred}{HTML}{CC1F1F}
\definecolor{green}{rgb}{.4,.7,.4}
\definecolor{blue}{rgb}{.2,.6,.75}
\definecolor{pastelb}{HTML}{3333FF}

\definecolor{pastelyellow}{rgb}{0.992157, 0.552941, 0.235294}
\definecolor{pastelorange}{rgb}{0.941176, 0.231373, 0.12549}
\definecolor{pastelred}{rgb}{0.741176, 0., 0.14902}
\definecolor{darkbrown}{rgb}{0.25098, 0., 0.0745098}

\usepackage{hyperref}
\hypersetup{pdftitle={Rigorous computation of the endomorphism algebra of a Jacobian},
pdfauthor={Edgar Costa, Nicolas Mascot, Jeroen Sijsling, and John Voight}}
\hypersetup{colorlinks=true,linkcolor=blue,anchorcolor=blue,citecolor=blue}

\begin{document}

\title[Endomorphism algebra computations]{Rigorous computation of the \\ endomorphism ring of a Jacobian}

\author{Edgar Costa}
\address{Department of Mathematics
77 Massachusetts Ave., Bldg. 2-252B
Cambridge, MA 02139, USA}
\email{edgarc@mit.edu}
\urladdr{\url{https://edgarcosta.org/}}

\author{Nicolas Mascot}
\address{Department of Mathematics
Faculty of Arts and Sciences
American University of Beirut
P.O. Box 11-0236
Riad El Solh,
Beirut 1107 2020
Lebanon}
\email{nm116@aub.edu.lb}
\urladdr{\url{https://staff.aub.edu.lb/~nm116/}}

\author{Jeroen Sijsling}
\address{Universit\"at Ulm, Institut f\"ur Reine Mathematik, D-89068 Ulm, Germany}
\email{jeroen.sijsling@uni-ulm.de}
\urladdr{\url{https://jrsijsling.eu/}}

\author{John Voight}
\address{Department of Mathematics, Dartmouth College, 6188 Kemeny Hall, Hanover, NH 03755, USA}
\email{jvoight@gmail.com}
\urladdr{\url{http://www.math.dartmouth.edu/~jvoight/}}

\begin{abstract}
    We describe several improvements and generalizations to algorithms for the
    rigorous computation of the endomorphism ring of the Jacobian of a curve
    defined over a number field.
\end{abstract}

\date{\today}

\maketitle

\section{Introduction}

\subsection{Motivation}

The computation of the geometric endomorphism ring of the Jacobian of a curve
defined over a number field is a fundamental question in arithmetic geometry.
For curves of genus $2$ over $\Q$, this was posed as a problem in 1996 by
Poonen \cite[\S 13]{poonen-ants}. The structure of the endomorphism ring and
its field of definition have important implications for the arithmetic of the
curve, for example when identifying of the automorphic realization of its
$L$-function \cite{bssvy}.

Let $F$ be a number field with algebraic closure $F^{\alg}$.  Let $X$ be a nice
curve over $F$, let $J$ be its Jacobian, and $J^{\alg}$ be its base change to
$F^{\alg}$. In this article, to \emph{compute the geometric endomorphism ring}
of $J$ means to compute an abstractly presented $\Z$-algebra $B$ (associative
with $1$ and free of finite rank as a $\Z$-module) equipped with a continuous
action of $\Gal(F^{\alg} \,|\, F)$ (factoring through a finite quotient)
together with a computable ring isomorphism
\begin{equation}
    \iota\colon B \xrightarrow{\sim} \End(J^{\alg})
\end{equation}
that commutes with the action of $\Gal(F^{\alg} \,|\, F)$.  (In this overview,
we are agnostic about how to encode elements of $\End(J^{\alg})$ in bits;
see below for a representation in terms of correspondences.)  Lombardo \cite[\S
5]{lombardo-endos} has shown that the geometric endomorphism ring can be
computed in principle using a day-and-night algorithm---but this algorithm
would be hopelessly slow in practice.

For a curve $X$ of genus $2$, there are practical methods to compute the
geometric endomorphism ring developed by van Wamelen \cite{vanwamelen-cm0,
vanwamelen-cm, vanwamelen-smart} for curves with complex multiplication (CM)
and more recently by Kumar--Mukamel \cite{kumar-mukamel} for curves with real
multiplication (RM).  A common ingredient to these approaches, also described
by Smith \cite{smith-thesis} and in its \Magma\ \cite{Magma} implementation by
van Wamelen \cite{vanwamelen-magma}, is a computation of the \emph{numerical
endomorphism ring}, in the following way.  First, we embed $F$ into $\C$ and by
numerical integration we compute a period matrix for $X$.  Second, we find
putative endomorphisms of $J$ by computing integer matrices (with small
coefficients) that preserve the lattice generated by these periods, up to the
computed precision.  Finally, from the tangent representation of such a
putative endomorphism, we compute a correspondence on $X$ whose graph is a
divisor $Y \subset X \times X$; the divisor $Y$ may then be rigorously shown to
give rise to an endomorphism $\alpha \in \End(J_K)$ over an extension $K
\supseteq F$ by exact computation. From this computation, we can also recover
the multiplication law in $\End(J^{\alg})$ and its Galois action \cite[\S
6]{bssvy}.

In the work on curves of genus $2$ of van Wamelen \cite{vanwamelen-cm} and
Kumar--Mukamel \cite{kumar-mukamel}, in the last step the divisor $Y$
representing the correspondence and endomorphism is found by interpolation, as
follows.  Let $P_0 \in X(F^{\alg})$ be a Weierstrass point on $X$.  Given a
point $P \in X(F^{\alg})$, by inverting the Abel--Jacobi map we compute the
(generically unique) pair of points $Q_1, Q_2 \in X(F^{\alg})$ such that
\begin{equation}\label{eq:alphaP}
    \alpha ([P - P_0]) = [Q_1 + Q_2 - 2 P_0] \in J^{\alg}
    =\Pic^0(X)(F^{\alg}).
\end{equation}
In this approach, the points $Q_1, Q_2$ are computed numerically, and the
divisor $Y$ is found by linear algebra by fitting $\{(P, Q_1),(P,Q_2)\} \subset
Y$ for a sufficiently large sample set of points $P$ on $X$.

\subsection{Contributions}

In this paper, we revisit this strategy and seek to augment its practical
performance in several respects. Our methods apply to curves of arbitrary genus
as well as isogenies between Jacobians, but we pay particular attention to the
case of the endomorphism ring of a curve of genus $2$ and restrict to this case
in the introduction. We present three main ideas which can be read
independently.

First, in section \ref{sec:complexendos}, we develop more robust numerical
infrastructure by applying methods of Khuri-Makdisi \cite{kkm-linalg} for
computing in the group law of the Jacobian.  Instead of directly inverting the
Abel--Jacobi map at point, we divide this point by a large power of $2$ to
bring it close to the origin where Newton iteration converges well, then we
multiply back using methods of linear series. In this way, we obtain increased
stability for computing the equality \eqref{eq:alphaP} numerically.

Second, in section \ref{sec:puiseux}, we show how to dispense entirely with
numerical inversion of the Abel--Jacobi map (the final interpolation step) by
working infinitesimally instead. Let $P_0 \in X(K)$ be a base point on $X$ over
a finite extension $K \supseteq F$. We then calculate the equality
\eqref{eq:alphaP} with $P=\widetilde{P}_0 \in X(K[[t]])$ the formal expansion
of $P_0$ with respect to a uniformizer $t$ at $P_0$. On an affine patch, we may
think of $\widetilde{P}_0$ as the local expansion of the coordinate functions
at $P_0$ in the parameter $t$. The points $Q_1$, $Q_2$ accordingly belong to a
ring of Puiseux series, and we can compute $Q_1$, $Q_2$ using a successive
lifting procedure with exact linear algebra to sufficient precision to fit the
divisor $Y$.  Our approach is similar to work of Couveignes--Ezome \cite[\S 6.2]{CE},
who also formally solve a differential system to compute equations for an 
isogeny between curves of genus $2$.  For completeness (and as a good warmup), 
we also consider in
section \ref{sec:newton} a hybrid method, where we compute \eqref{eq:alphaP}
for a single suitable point $P \neq P_0$ and then successively lift over a ring
of power series instead. In both cases, we obtain further speedups by working
over finite fields and using a fractional version of the Chinese remainder
theorem. These methods work quite well in practice.

Third, in section \ref{sec:upperbounds} we consider upper bounds on the
dimension of the endomorphism algebra as a $\Q$-vector space, used to match the
lower bounds above and thereby sandwiching the endomorphism ring.  Lombardo
\cite[\S 6]{lombardo-endos} has given such upper bounds in genus $2$ by
examining Frobenius polynomials; we consider a slightly different approach in
this case by first bounding from above the dimension of the subalgebra of
$\End(J^{\alg})_{\Q}$ fixed under the Rosati involution (using the known
Tate conjecture for the reduction of the abelian surface modulo primes).  This
specialized algorithm in genus 2 again is quite practical.  We then generalize
this approach to higher genus: applying work of Zywina \cite{zywina-14}, we
again find rigorous upper bounds and we show that these are sharp if the
Mumford--Tate conjecture holds for the Jacobian and if a certain hypothesis
on the independence of Frobenius polynomials holds.

We conclude in section \ref{sec:examples} with some examples. Confirming
computations of Lombardo \cite[\S 8.2]{lombardo-endos}, we also verify the
correctness of the endomorphism data in the \emph{$L$-functions and Modular
Forms DataBase} (LMFDB) \cite{lmfdb} which contains 66\,158 curves of genus $2$
with small minimal absolute discriminant.

Our implementation of these results is available online \cite{cms-package}, and
all examples in this paper can be inspected in detail by going to its
subdirectory \texttt{endomorphisms/examples/paper}. This code has already been
used by Cunningham--Demb\'el\'e to establish the paramodularity of an abelian
threefold in the context of functoriality \cite{cunningham-dembele-hmf}.

\subsection*{Acknowledgments}

The authors would like to thank Kamal Khuri-Makdisi and David Zywina for
helpful conversations, as well as the anonymous referees for their comments and
suggestions.  Mascot was supported by the EPSRC Programme Grant
EP/K034383/1 ``LMF: L-Functions and Modular Forms''. Sijsling was supported by
the Juniorprofessuren-Programm ``Endomorphismen algebraischer Kurven''
(7635.521(16)) from the Science Ministry of Baden--W\"urttemberg. Voight was
supported by an NSF CAREER Award (DMS-1151047) and a Simons Collaboration Grant
(550029).

\section{Setup} \label{sec:numexactendos}

To begin, we set up some notation and background, and we discuss
representations of endomorphisms in bits.

\subsection{Notation}

Throughout this article, we use the following notation. Let $F \subset \C$ be a
number field with algebraic closure $F^{\alg}$.  Let $X$ be a nice (i.e.,
smooth, projective and geometrically integral) curve over $F$ of genus $g$. Let
$J=\Jac(X)$ be the Jacobian of $X$.  We abbreviate $J^{\alg}=J_{F^{\alg}}$ for
the base change of $J$ to $F^{\alg}$.  When discussing algorithms, we assume
that $X$ is presented in bits by equations in affine or projective space; by
contrast, we will not need to describe $J$ as a variety defined by equations,
as we will only need to describe the points of $J$.

\subsection{Numerical endomorphisms}

The first step in computing the endomorphism ring is to compute a numerical
approximation to it. This technique is explained in detail by van Wamelen
\cite{vanwamelen-magma} in its \Magma\ \cite{Magma} implementation for
hyperelliptic curves. See also the sketch by
Booker--Sijsling--Sutherland--Voight--Yasaki \cite[\S 6.1]{bssvy} where with a
little more care the Galois structure on the resulting approximate endomorphism
ring is recovered as well.

The main ingredients of the computation of the numerical endomorphism ring are
the computation of a period matrix of $X$---i.e., the periods of an $F$-basis
$\omega_1,\dots,\omega_g$ of the space of global differential $1$-forms on $X$
over a chosen symplectic homology basis---followed by lattice methods. (For
more detail on period computations, see the next section.) The output of this
numerical algorithm is a putative $\Z$-basis $R_1, \dots, R_d \in \M_{2g}(\Z)$
for the ring $\End (J^{\alg})$. These matrices represent the action of the
corresponding endomorphisms on a chosen basis of the homology group $H_1 (X,
\Z)$, and accordingly, the corresponding ring structure is induced by matrix
multiplication. If $\Pi \in \M_{g, 2g}(\C)$ is the period matrix of $J$, then
the equality
\begin{equation}
    \label{equation:analytic_geometric}
    M \Pi = \Pi R
\end{equation}
holds, where $M \in \M_g(\C)$ is the representation on the tangent space $H^0
(X, \omega_X)^*$, given by left multiplication.  Equation
\eqref{equation:analytic_geometric} allows us to convert (numerically) between
the matrices $R_j \in \M_{2g}(\Z)$ and matrices $M_j \in \M_g(\C)$ describing
the action on the tangent space, which allows us to descend to $\M_g
(F^{\alg})$ and hence to $\M_g (K)$ for extensions of $K$ by using Galois
theory.

We take this output as being given for the purposes of this article; our goal
is to certify its correctness.

\begin{remark}
    In other places in the literature, equation
    \eqref{equation:analytic_geometric} is transposed. We chose this convention
    because it makes the map $\End (J) \to \End( H^0(X,\omega_X)^*)$ a
    ring homomorphism.
\end{remark}

\begin{example} \label{example:QMexample-numeric}
    We will follow one example throughout this paper, followed by several other
    examples in the last section.

    Consider the genus 2 curve
    $X \colon y^2 = x^{5} - x^{4} + 4 x^{3} - 8 x^{2} + 5 x - 1$ with LMFDB
    label
    \href{http://www.lmfdb.org/Genus2Curve/Q/262144/d/524288/1}{\QMtwosix}.
    As described above, we find the period matrix
    \begin{equation}
        \Pi \approx
        \begin{pmatrix}
            1.851 - 0.1795 i & 3.111 + 2.027 i & -1.517 + 0.08976 i & 1.851 \\
            0.8358 - 2.866 i & 0.3626 + 0.1269 i & -1.727 + 1.433 i & 0.8358
        \end{pmatrix}
    \end{equation}
    (computed to 600 digits of precision in about 10 CPU seconds on a standard
    desktop machine).  We then verify that $X$ has numerical quaternionic
    multiplication. More precisely, we have numerical evidence that endomorphism
    ring is a maximal order in the quaternion algebra over $\Q$ with discriminant
    $6$. For example, we can identify a putative endomorphism $\alpha
    \overset{?}{\in} \End(J_\C)$ with representations
    \begin{equation}
        M = \left(\begin{array}{rr}
                0 & \sqrt{2} \\
                \sqrt{2} & 0
        \end{array}\right)
        \text{ and }
        R =
        \begin{pmatrix}
            0 & -3 & 0 & -1 \\
            -2 & 0 & 1 & 0 \\
            0 & -4 & 0 & -2 \\
            4 & 0 & -3 & 0
        \end{pmatrix} ,
    \end{equation}
    which satisfies $\alpha^2 = 2$.
\end{example}

The numerical stability of the numerical method outlined above has not been
analyzed. The \Magma\ implementation will occasionally throw an error because
of intervening numerical instability (see Example \ref{exm:magmaproblem}
below); this can often be resolved by slightly transforming the defining
equation of $X$.

\begin{remark}
  There are several available implementations to compute the period
  matrix and the Abel--Jacobi map in addition to
  \Magma. A recent robust method to calculate period matrices of cyclic covers
  of the projective line was developed by Molin--Neurohr \cite{molin-neurohr}.
  We also recommend the introduction of this reference for a survey of other
  available implementations.

  Work continues: Neurohr is working on the generalization of
  these algorithms to (possibly singular) plane models of general algebraic
  curves, and for these curves a \SageMath\ implementation by Nils Bruin and
  Alexandre Zotine is also in progress.
\end{remark}

\begin{remark}
  For hyperelliptic curves and plane quartics we may also speed up the
  calculation of periods through arithmetic--geometric mean (AGM) methods. So
  far this has been implemented in the hyperelliptic case \cite{sijsling-agm}.
  While this delivers an enormous speedup, the AGM method introduces a change
  of basis of differentials, which makes us lose information regarding the
  Galois action.
\end{remark}

\section{Complex endomorphisms} \label{sec:complexendos}

In this section, we describe a numerically stable method for inversion of the
Abel--Jacobi map.

\subsection{Abel--Jacobi setup}

Let $P_0 \in X(\C)$ be a base point and let
\begin{equation}
    \begin{aligned}
        \AJ_{P_0} \colon X &\to J \\
        P &\mapsto [P-P_0]
    \end{aligned}
\end{equation}
be the Abel--Jacobi map associated to $P_0$. Complex analytically, using our
chosen basis $\omega_1, \ldots , \omega_g$ of $H^0 (X, \omega_X)$
we identify
$J(\C) \simeq \C^g/\Lambda$ where $\Lambda \simeq \Z^{2g}$ is the period
lattice of $J$. Under this isomorphism the Abel--Jacobi map is
\begin{equation}
    \AJ_{P_0}(P)
    =
    \left(\int_{P_0}^{P} \omega_i\right)_{i=1,\dots,g} \in \C^g/\Lambda .
\end{equation}
The numerical evaluation of these integrals is standard: we compute a low
degree map $\varphi\colon X \to \PP^1$, make careful choices of the branch cuts
of $\varphi$, and then integrate along a polygonal path that avoids the
ramification points of $\varphi$.

\begin{example} \label{exm:hyperell}
    Suppose $X$ is a hyperelliptic curve of genus $g$ given by an equation of the
    form $y^2 = f(x)$ where $f(x)$ is squarefree of degree $2g+1$ or $2g+2$.
    Then an $F$-basis of differentials is given by
    \begin{equation}
        \omega_1 =\frac{\d{x}}{y},\
        \omega_2 = x\frac{\d{x}}{y},\ldots,\
        \omega_g = x^{g-1}\frac{\d{x}}{y}.
    \end{equation}
    In the $x$-plane, we draw a polygonal path $\gamma_x$ from $x(P_0)$ to
    $x(P)$ staying away from the roots of $f(x)$ different from $P_0, P$. We then
    lift $\gamma_x$ to a continuous path $\gamma$ on $X$.

    Suppose for simplicity that $P_0$ is not a Weierstrass point, so $f(x(P_0))
    \neq 0$.  (The case where $P_0$ is a Weierstrass point can be handled
    similarly by a choice of square root and more careful analysis.) Then
    $y(P_0) = \sqrt{f(x(P_0))}$ selects a branch of the square root. To keep
    track of the square root along $\gamma$, we use four determinations of the
    square root over $\C$, with respective branch cuts along the half-axes $\Re
    z >0$, $\Re z < 0$, $\Im z > 0$ and $\Im z < 0$. On each segment of
    $\gamma_x$, we change the branch of the square root whenever $\Re f$ or
    $\Im f$ changes sign, so as to keep the branch cut away from the values of
    $f(x)$. For instance, in the case illustrated by Figure
    \ref{fig:beautifulpic}, letting $t$ be the parameter of integration and
    assuming we started with the determination whose branch cut is along $\Im z
    > 0$, we would first switch to the determination whose branch cut is along
    $\Re z < 0$ when $\Im f(\gamma_x(t))$ changes from negative to positive,
    and then to the determination whose branch cut is along $\Im z < 0$ when
    $\Re f(\gamma_x(t))$ changes from positive to negative, so that the branch
    cut is always at least $90^\circ$ away from $f(\gamma_x(t))$. Of course,
    the sign of the square root may need to be corrected every time we switch
    from one determination to another, so as to get a continuous determination
    of $\sqrt{f(\gamma(t))}$. Also note that by construction, the integration
    path avoids the roots of $f$, so the signs of $\Re f(\gamma(t))$ and $\Im
    f(\gamma(t))$ never change simultaneously.

    \addtocounter{equation}{1}
    \begin{figure}[t]
        \begin{subfigure}[b]{0.30\textwidth}
            \scalebox{0.7}{
                \begin{tikzpicture}[scale=1, baseline=(current bounding box.center)]
                    \draw [domain=4:8, scale=1] plot (20*\x:\x);
                    \draw[domain=-3:3.5,scale=1] plot(\x-3,\x*\x/5-\x+3);
                    \draw (-4.46,4.88) node {\textbullet};
                    \draw (-5.4,4.95) node {$f=0 \ $};
                    \draw (-0.5,1.4) node {$\Im f < 0$};
                    \draw (-0.5,2.2) node {$\Im f > 0$};
                    \draw (0,4) node[rotate=-30] {$\Re f > 0$};
                    \draw (0.4,4.7) node[rotate=-30] {$\Re f < 0$};
                    \draw[very thick, pastelorange, ->] (-7,1) -- (-5,2.5);
                    \draw [very thick, pastelorange] (-5,2.5) -- (-3.55,3.59);
                    \draw[very thick, pastelyellow, ->]  (-3.55,3.59) -- (-2.5,4.375);
                    \draw[very thick, pastelyellow]  (-2.5,4.375) -- (-1.5,5.1);
                    \draw[very thick, pastelred, ->]  (-1.5,5.1) -- (0,25/4);
                    \draw[very thick, pastelred]  (0,25/4) -- (1,7);
                    \draw (-6,2.2) node[rotate=37] {Integration path $\gamma$};
                \end{tikzpicture}
            }
        \end{subfigure}
        \begin{subfigure}{0.3\textwidth}
            \begin{tikzpicture}[scale=0.7, baseline=(current bounding box.center)]
                \draw (0,0);
                \draw[ ->] (4,-1)--(6,-1);
                \draw (5,-0.3) node {$f$};
            \end{tikzpicture}
        \end{subfigure}
        \begin{subfigure}[b]{0.30\textwidth}
            \scalebox{0.7}{
                \begin{tikzpicture}[scale=1, baseline=(current bounding box.center)]
                    \draw[->] (-3,0) -- (3,0) node[right] {$\Re f$};
                    \draw[->] (0,-3) -- (0,3) node[above] {$\Im f$};
                    \draw[very thick, pastelorange, ->, domain=-1.2:-0.6,scale=1] plot(2-\x*\x,1.5*\x);
                    \draw[very thick, pastelorange, domain=-0.6:0,scale=1] plot(2-\x*\x,1.5*\x);
                    \draw[very thick, pastelyellow, ->, domain=0:0.7,scale=1] plot(2-\x*\x,1.5*\x);
                    \draw[very thick, pastelyellow, domain=0.7:1.414,scale=1] plot(2-\x*\x,1.5*\x);
                    \draw[very thick, pastelred, ->, domain=1.414:1.8,scale=1] plot(2-\x*\x,1.5*\x);
                    \draw[very thick, pastelred, domain=1.8:2,scale=1] plot(2-\x*\x,1.5*\x);
                \end{tikzpicture}
            }
        \end{subfigure}
        \caption{Changing the branches of $\sqrt{f(x)}$ along
        $\gamma$}\label{fig:beautifulpic}
    \end{figure}
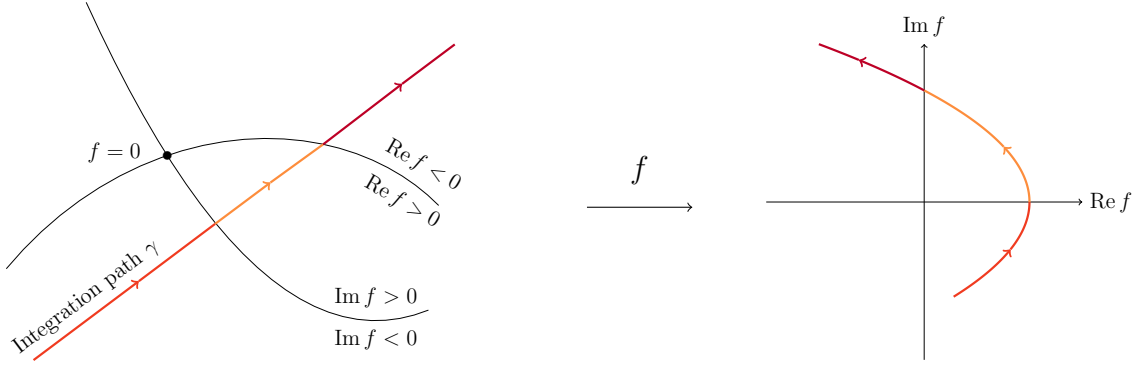
    In this way, the integrals $\int_{P_0}^P \omega_j$ can be computed, and
    thereby the Abel--Jacobi map.
\end{example}

Now let $O_0=O_{0,1}+\dots+O_{0,g}$ be an effective (``origin'') divisor of
degree $g$. Riemann--Roch ensures that for a generic choice of pairwise
distinct points $O_{0,k} \in X(\C)$, the derivative of the Abel--Jacobi map
\begin{equation}
    \begin{aligned}
        \AJ \colon \Sym^g(X)(\C) & \to \C^g/\Lambda \\
        \{ Q_1, \cdots, Q_g \} & \mapsto \displaystyle \sum_{k=1}^g
        \left( \int_{O_{0,k}}^{Q_k} \omega_j \right)_{j=1,\cdots,g}
    \end{aligned}
\end{equation}
is non-singular at $O_0$, so we assume that this is indeed the case from now
on. As explained by Mumford, for a general point $[D] \in J(\C)=\Pic^0(X)(\C)$,
by Riemann--Roch we can write
\begin{equation}
    [D] = [Q_1 + \dots + Q_g - O_0]
\end{equation}
with $Q_1,\dots,Q_g \in X(\C)$ unique up to permutation; this defines a
rational map
\begin{equation}
    \begin{aligned}
        \Mum\colon J &\dashrightarrow \Sym^g(X) \\
        [D] &\mapsto \{Q_1,\dots,Q_g\}.
    \end{aligned}
\end{equation}
The composition $\AJ \circ \Mum$ is the identity map on $J$, so then $\Mum$ is
a right inverse to $\AJ$.  Analytically, for $b \in \C^g/\Lambda$, we have
$\Mum(b) = \{Q_1,\dots,Q_g\}$ where
\begin{equation} \label{eqn_AJ}
    \left(\sum_{k=1}^g \int_{O_{0,k}}^{Q_k} \omega_j \right)_{j=1,\dots,g}
    \equiv b \pmod{\Lambda}.
\end{equation}

Now let $\alpha \in \End(J_\C)$ be a nonzero numerical endomorphism represented
by the matrix $M \in \M_g(\C)$ as in \eqref{equation:analytic_geometric}.
Consider the following composed rational map
\begin{equation}\label{eqn:alphastar}
    \alpha_X
    \colon
    X \xrightarrow{\AJ} J \xrightarrow{\alpha} J \xdashrightarrow{\Mum} \Sym^g(X).
\end{equation}
Then we have $\alpha_X(P)=\{Q_1,\dots,Q_g\}$ if and only if
\begin{equation}
    \alpha([P - P_0])=[Q_1 + \dots + Q_g - O_0].
\end{equation}
As mentioned in the introduction, the map $\alpha_X$ can be used to rigorously
certify that $\alpha$ is an endomorphism of $J$ by interpolation.  We just saw
how to compute the Abel--Jacobi map via integration, and the application of
$\alpha$ amounts to matrix multiplication by $M$.  So the tricky aspect is in
computing the map $\Mum$, inverting the Abel--Jacobi map.  We will show in the
next subsections how to accomplish this task in a more robust way than by naive
inversion.

\subsection{Algorithms of Khuri-Makdisi}

Our method involves performing arithmetic in $J$, and for this purpose we use
algorithms developed by Khuri-Makdisi \cite{kkm-linalg}.  Let $D_0 \in
\Div(X)(\C)$ be a divisor of degree $d_0 > 2g$ on $X$.  By Riemann--Roch, every
class in $\Pic^0(X)(\C)$ is of the form $[D-D_0]$ where $D \in \Div(X)(\C)$ is
effective of degree $d_0$.  We represent the class $[D-D_0]$ by the subspace
\begin{equation}
    W_D \colonequals H^0(X,3 D_0-D) \subseteq V \colonequals H^0(X,3D_0).
\end{equation}
The divisor $D$ is usually not unique, hence neither is this representation of
a class in $\Pic^0(X)(\C)$ as a subspace of $V$. However, Khuri-Makdisi has
exhibited a method \cite[Proposition/Algorithm 4.3]{kkm-linalg} that, given as
input two subspaces $W_{D_1}$ and $W_{D_2}$ representing two classes in
$\Pic^0(X)(\C)$, computes as output a subspace $W_{D_3}$ corresponding to a
divisor $D_3$ such that $D_1 + D_2 + D_3 \sim 3 D_0$ by performing linear
algebra in the spaces $V$ and $V_2 \colonequals H^0(X,6 D_0)$.  In this way, we
can compute explicitly with the group law in $J$.

\begin{example}\label{exm:hyer_RR}
    Suppose $X$ is as in Example \ref{exm:hyperell}.  We find a basis for $V$ and
    $V_2$ as follows.  A natural choice for $D_0$ is $(g+1) \Infty$, where
    $\Infty=\pi^{-1}(\infty)$ is the preimage of $\infty \in \PP^1$ under the
    hyperelliptic map $x\colon X \to \PP^1$.  If $f$ has even degree, then
    $\Infty$ is the sum of two distinct points; if $f$ has odd degree, then
    $\Infty$ is twice a point.  In either case, the divisor $(g-1)\Infty$ is a
    canonical divisor on $X$, and $\deg \infty_X=2$; by Riemann--Roch for $m \geq
    g+1$ the space $H^0(X,m\Infty)$ has basis given by $1, x, \dots, x^m, y, xy,
    \dots, x^{m-g-1}y$.
\end{example}

In what follows, we represent functions in $V_2 \supsetneq V$ by their
evaluation at any $N > 6 d_0$ points of $X(\C)$ disjoint from the support of
$D_0$.

\subsection{Inverting the Abel--Jacobi map} \label{sec:invertabeljac}

Let $b \in \C^g/\Lambda$ correspond to a divisor class $[C] \in \Pic^0(X)(\C)$;
for example, $b=M \AJ(P)$ for $P \in X(\C)$ and $M$ representing a putative
endomorphism.  We now explain how to compute $\Mum(b)=\{Q_1,\dots,Q_g\}$ as in
\eqref{eqn_AJ}, under a genericity hypothesis.

If we start with arbitrary values for $Q_1,\dots,Q_g$, we can adjust these
points by Newton iteration until equality is satisfied to the desired
precision. However, there are no guarantees on the convergence of the Newton
iteration!

\medskip
\emph{Step 1: Divide the point and Newton iterate.}
Following Mascot \cite[\S 3.5]{mascot-palermo}, we first replace $b$ with a
point $b'$ very close to $0$ modulo $\Lambda$ and such that $2^m b' \equiv b
\psmod{\Lambda}$ for some $m \in \Z_{\geq 0}$.  For example, $b'$ may be
obtained by lifting $b$ to $\C^g$ and dividing the resulting vector by $2^m$.

As $b'$ is very close to $0$ modulo $\Lambda$, the equation \eqref{eqn_AJ}
should have a solution $\{Q_k'\}_j$ with $Q_k'$ close to $O_{0,k}$ for
$k=1,\dots,g$ since the derivative of the Abel--Jacobi map $\AJ$ at $O_0$ is
nonsingular by assumption.  We start with $Q_k' = O_{0,k}$ as initial guesses,
and then use Newton iteration until \eqref{eqn_AJ} holds to the desired
precision.  If Newton iteration does not seem to converge, we increase the
value of $m$ and start over.  The probability of success of the method
described above increases with $m$.  In practice, we found that starting with
$m = 10$ was a good compromise between speed and success rate.

In this way, we find points $Q_1',\dots,Q_g'$ such that the linear equivalence
\begin{equation}
    C \sim 2^m \left(\sum_{k=1}^g Q_k' - O_0 \right)
\end{equation}
holds in $\Div(X)^0(\C)$.

\medskip
\emph{Step 2: Recover the divisor by applying an adaptation of the
Khuri-Makdisi algorithm.}
From this, we want to compute $Q_1,\dots,Q_g$ such that
\begin{equation}
    C \sim \sum_{k=1}^g Q_k - O_0.
\end{equation}
For this purpose, we work with divisors and the algorithms of the previous
section.  But these algorithms only deal with divisor classes of the form
$[D-D_0]$ with $\deg D=d_0$ whereas we would like to work with $[\sum_{k=1}^g
Q_k' - O_0]$.  So we adapt the algorithms in the following way.

We choose $d_0-g$ auxiliary points $P_1, \dots, P_{d_0-g} \in X(\C)$ distinct
from the points $Q'_k$, the points $O_{0,k}$, and the support of $D_0$.
Consider the divisors
\begin{equation}
    \begin{aligned}
        D_+ &\colonequals \sum_{k=1}^g Q'_k + \sum_{k=1}^{d_0-g} P_k \\
        D_- &\colonequals O_0 + \sum_{k=1}^{d_0-g} P_k,
    \end{aligned}
\end{equation}
both effective of degree $d_0$.  We then compute the subspaces $W_{D_+}$ and
$W_{D_-}$ of $V$, and apply the subtraction algorithm of Khuri-Makdisi: we
obtain a subspace $W_{D'}$ corresponding to an effective divisor $D'$ such that
\begin{equation}
    D'-D_0 \sim \left(\sum_{k=1}^g Q_k' + \sum_{k=1}^{d_0-g} P_k \right)
    - \left(O_0 + \sum_{k=1}^{d_0-g} P_k \right)
    = \sum_{k=1}^g Q_k' - O_0.
\end{equation}
We then repeatedly use the doubling algorithm to compute $W_{D}$, where $D$ is
a divisor such that $D-D_0 \sim 2^m (D'-D_0)$. We have thus computed a subspace
$W_D$ such that
\begin{equation}
    D - D_0 \sim C \sim \sum_{k=1}^g Q_k - O_0. \label{eq:Qi1}
\end{equation}

To conclude, we recover the points $Q_1,\dots,Q_g$ from $W_D$ in a few more
steps.  We proceed as in Mascot \cite[\S 3.6]{mascot-palermo}.

\medskip
\emph{Step 3: Compute $E \sim \sum_k Q_k$.}
We apply the addition algorithm to $W_D$ and $W_{D_-}$ and negate the result.
(In fact, Khuri-Makdisi's algorithm computes these two steps in one.)  This
results in a subspace $W_\Delta$ where $\Delta$ is an effective divisor with
$\deg \Delta=d_0$ and
\begin{equation}
    \Delta-D_0 \sim (D_0 - D) + (D_0 - D_-).
\end{equation}
By \eqref{eq:Qi1}, we have
\begin{equation}
    \sum_{k=1}^g Q_k \sim E, \quad
    \text{where } E \colonequals
    2 D_0 - \Delta - \sum_{k=1}^{d_0-g} P_k \label{eq:Qi2}
\end{equation}
and $\deg(E)=g$.

\medskip
\emph{Step 4: Compute $Z=H^0(X,E)$.}  Next, we compute
\begin{equation}
    H^0(X,3 D_0 - \Delta) \cap H^0(X, 2 D_0)
\end{equation}
and the subspace $Z$ of this intersection of functions that vanish at all
$P_k$.  Generically, we have
\begin{equation}
    Z = H^0(X,E)
\end{equation}
and since $\deg(E)=g$, by Riemann--Roch we have $\dim Z \geq 1$.  The
genericity assumption may fail, but we can detect its failure by comparing the
(numerical) dimension of the resulting spaces with the value predicted by
Riemann--Roch, and rectify its failure by restarting with different auxiliary
points $P_k$.

\medskip
\emph{Step 5: Recover the points $Q_i$.}  Now let $z \in Z$ be nonzero; then
\begin{equation}\label{eq:divs}
    \opdiv z = Q-E
\end{equation}
where $Q$ is an effective divisor with $\deg Q=g$ and
\begin{equation}
    Q \sim \sum_{k=1}^g Q_k
\end{equation}
by \eqref{eq:Qi2}; as we are always working up to linear equivalence, we may
take $Q = \sum_{k=1}^g Q_k$ as desired.  To compute $\opdiv z$ and
circumnavigate the unknown divisor $\Delta$, we compute the subspace
\begin{equation}
    Z' \colonequals \big\{ v \in V : v W_{\Delta} \subseteq z V \big\}
\end{equation}
where $z V = H^0(X, 3 D_0 - \opdiv z)$ and $W_\Delta = H^0(X,3 D_0 - \Delta)$.
Since $3 D_0 - \Delta$ is basepoint-free (its degree exceeds $2g$), we conclude
that
\begin{equation}
    Z' = H^0\big(X,3 D_0 - \opdiv z - (3 D_0 - \Delta) \big)
    =
    H^0\left(X, 2 D_0 - \sum_{k=1}^{d_0-g} P_k  - \sum_{k=1}^g Q_k \right).
\end{equation}
We then recover the divisor $\sum_k P_k + \sum_k Q_k$ as the intersection of
the locus of zeros of the functions in $Z'$, and then the points $Q_k$
themselves whenever they are distinct from the chosen auxiliary points $P_k$.
Once more, this procedure works for generic input, and we can check if
we are in the generic case and rectify failure if this turns out not to be the
case.

\begin{example}
    In the case of a hyperelliptic curve, as in Example \ref{exm:hyer_RR} with
    $D_0 = (g+1) \infty_X$, the method described above leads us to
    \begin{equation}
        T = H^0\left(X, (2g+2) \infty_X - \textstyle{\sum}_{k=1}^{d_0-g} P_k
        - \textstyle{\sum}_{k=1}^g Q_k \right),
    \end{equation}
    which consists of functions which are linear combinations of $x^n$ and $x^n
    y$ for $n \in \Z_{\geqslant 0}$. These linear combinations thus describe
    polynomial equations that the coordinates of the points $P_k$ and $Q_k$ must
    satisfy, which allows us to recover the $Q_k$.
\end{example}

\begin{remark}
    Khuri-Makdisi's method relies only linear algebra operations in vector
    spaces of dimension $O(g \log g)$.  As we are working numerically, we must
    rely upon numerical linear algebra, and in our implementation we performed
    most of these operations by QR decompositions, a good trade-off between
    speed and stability. In practice, our loss of precision was at most 10
    precision bits per Jacobian operation.
\end{remark}

\subsection{Examples}

We now present two examples of the above approach.

\begin{example} \label{example:QMexample-IAJ}
    We return to Example \ref{example:QMexample-numeric}. Let $P_0 =\left(1,
    0\right)$ and $ P = (2, 5) $.
    Integrating, we find $ \AJ_{P_0}(P)  \equiv b
    \pmod{\Lambda}$ where
    \begin{equation}
        b \approx \left( 0.2525 , 1.475 \right),
    \end{equation}

    We now apply the methods of section \ref{sec:invertabeljac}. We arbitrarily
    set
    \begin{equation}
        \begin{aligned}
            O_{0,1} &= \left(0.9163 + 0.8483i,\,1.104 - 1.884i\right),\\
            O_{0,2} &= \left(0.3311 + 0.9656i,\,2.159 - 0.3835i\right).
        \end{aligned}
    \end{equation}
    The first step inverts the Abel--Jacobi map to obtain
    \begin{equation}
        2^{-10} M b = \AJ(\{Q_1' , Q_2'\} )
    \end{equation}
    where
    \begin{equation}
        \begin{aligned}
            Q_1' &\approx \left(0.9224 + 0.8521i,\,1.103 - 1.909i\right),\\
            Q_2' &\approx \left(0.3257 + 0.9592i,\,2.146 - 0.3645i\right).\\
        \end{aligned}
    \end{equation}
    The remaining steps (adapting the algorithms of Khuri-Makdisi) compute
    $Q_1$ and $Q_2$ such that
    \begin{equation}
        2^{10}[ Q_1' + Q_2' - O_{0,1} - O_{0,2}] = [Q_1 + Q_2 -  2 P_0],
    \end{equation}
    where
    \begin{equation}
        Q_k \approx \left(0.7500 \pm 0.4330i,\,-0.4419 \pm 0.7655i\right) .
    \end{equation}

    Using the LLL algorithm \cite{LLL}, we guess that the $x$-coordinates of
    $Q_1$ and $Q_2$ satisfy $4x^{2} - 6 x + 3=0$, and under this assumption we
    have
    \begin{equation}
        Q_k =
        \left(
            \frac{ 3 \pm i \sqrt{3}}{4}
            ,
            \frac{- 5\sqrt{2} \pm 5 i \sqrt{6}}{16}
        \right).
    \end{equation}
    All the computations above were performed with at least 600 decimal digits.
    On a standard desktop machine, figuring out right number of points for the
    Gauss--Legendre quadrature and calculating $b$ took less than $3$ CPU seconds,
    and the computation of the points $Q_1$ and $Q_2$ took around $2$ CPU minutes.
\end{example}

\begin{example} \label{exm:magmaproblem}
    The \Magma\ functions \texttt{ToAnalyticJacobian} and
    \texttt{FromAnalyticJacobian} provide us similar functionality. However, we
    have found these algorithms to be sometimes numerically unstable (in
    v2.22-6).

    For example, consider the curve with LMFDB label
    \href{http://www.lmfdb.org/Genus2Curve/Q/169/a/169/1}{\texttt{169.a.169.1}},
    a model for the modular curve $X_1 (13)$ with equation
    \begin{equation}
        X\colon y^2 = x^{6} + 4 x^{5} + 6 x^{4} + 2 x^{3} + x^{2} + 2 x + 1.
    \end{equation}
    We find a numerical endomorphism $\alpha$ with $\alpha^2=1$ defined over
    $\Q(\lambda)$ where $\lambda=2\cos(2\pi/13)$, with matrix
    \begin{equation}
        M = \frac{1}{13}\begin{pmatrix}
            -7\lambda^5 - 8\lambda^4 + 32\lambda^3 + 27\lambda^2 - 27\lambda - 10 &
            -5\lambda^5 - 2\lambda^4 + 21\lambda^3 + 10\lambda^2 - 10\lambda - 9 \\
            2\lambda^5 + 6\lambda^4 - 11\lambda^3 - 17\lambda^2 + 17\lambda + 1 &
            7\lambda^5 + 8\lambda^4 - 32\lambda^3 - 27\lambda^2 + 27\lambda + 10
        \end{pmatrix}.
    \end{equation}
    For a random point $P$, \Magma\ is unable to compute
    \begin{equation*}
        \texttt{FromAnalyticJacobian}\left(
            \alpha \cdot
        \texttt{ToAnalyticJacobian}(P, X), X\right)
    \end{equation*}
    in precision $600$. A workaround in this case is to replace $\alpha$ by
    $\alpha + 1$ instead; it is unclear why such a modification restores
    numerical stability (sometimes a change of variables in the equation also
    suffices).

    In comparison, if we set $P_0=(0,1)$, $P_1=(-1,1)$, and
    $O_{0}=\{ \infty_{+}, \infty_{-} \}$,
    then thanks to the above approach we can compute that
    \begin{equation}
        M \AJ_{P_0}(P_1) = \AJ(\{Q_1, Q_2\}),
    \end{equation}
    or in other words
    \begin{equation}
        \alpha([P_1-P_0]) = [Q_1+Q_2-\infty_+-\infty_-],
    \end{equation}
    where
    \begin{equation}
        \begin{aligned}
            Q_1& \approx (-1.3772, 1.8730), \\
            Q_2& \approx (2.6511, 34.8995). \\
        \end{aligned}
    \end{equation}
    With 600 decimal digits of accuracy, the computation takes about $1$ CPU
    minute.

    The LLL algorithm then suggests that
    \begin{equation}
        \begin{aligned}
            Q_1& = (\theta^2+2\theta-2, \ 11 \lambda^5+18\lambda^4
            -43\lambda^3-66\lambda^2+26\lambda+33), \\
            Q_2& = (-\theta^2-\theta+3, \ -6\lambda^5+6\lambda^4
            +31\lambda^3-19\lambda^2-21\lambda+5), \\
        \end{aligned}
    \end{equation}
    where $\theta = \lambda^{5} - 5 \lambda^{3} + 6 \lambda$, which holds to at least 500 decimal places.

\end{example}

\begin{remark}
    In the example above, it is surprising that $Q_1$ and $Q_2$ are both
    defined (instead of being conjugate) over $\Q(\lambda)$ and that their
    $x$-coordinates are defined over the subfield $\Q(\theta)$. This happens
    because $\alpha$ turns out to be induced by a modular (sometimes called a
    \emph{Fricke}) involution of $X_1(13)$ (to be precise, the one attached to
    the root of unity $e^{8 \pi i / 13}$), and because $X_1(13)(\Q)$ only
    contains cusps, so that $P_0$, $P_1$, $\infty_+$, $\infty_-$ and thus $Q_1$
    and $Q_2$ are cusps.
\end{remark}

\section{Newton lift}
\label{sec:newton}

In the previous section, we showed how one can numerically compute the
composite map
\[     \alpha_X
    \colon
    X \xrightarrow{\AJ_{P_0}} J \xrightarrow{\alpha}
    J \xdashrightarrow{\Mum} \Sym^g(X). \]
given $\alpha \in \End(J_\C)$. As explained in the introduction, by
interpolation we can then fit a divisor $Y \subset X \times X$ representing the
graph of the numerical endomorphism $\alpha$. When this divisor is defined over
a number field and the induced homomorphism on differentials as in Smith
\cite[\S 3.5]{smith-thesis} is our given tangent matrix, then we have
successfully verified the existence of the corresponding endomorphism. In this
section---one that can be read as a warmup for the next section or as a hybrid
method---we only use numerical approximation for a single point, after which we
use a Newton lift to express the endomorphism in a formal neighborhood.

\subsection{Setup}

We retain the notation of the previous section.  We further suppose that the
base point $P_0 \in X(K)$ and origin divisor $O_0=\sum_{i=1}^g O_{0,i} \in
\Div^0(X)(K)$ are defined over a finite extension $K \supseteq F$.  Enlarging
$K$ further if necessary, we choose $P \in X(K)$ distinct from $P_0$ and
suppose (as computed in the previous section, or another way) that we are given
points $Q_1,\dots,Q_g \in X(K)$ such that numerically we have
\begin{equation}
    \alpha_X(P) = \{Q_1,\dots,Q_g\}.
\end{equation}
Moreover, possibly enlarging $K$ again, we may assume the matrix $M$
representing the action of $\alpha$ on differentials has entries in $K$.

For concreteness, we will exhibit the method for the case of a hyperelliptic
curve; we restore generality in the next section.  Suppose $X\colon y^2=f(x)$ is
hyperelliptic as in Example \ref{exm:hyperell}.  Let $t \colonequals x-x(P)$;
we think of $t$ as a formal parameter.  We further assume that $t$ is a
\emph{uniformizer} at $P$: equivalently, $f(x(P)) \neq 0$, i.e., $P$ is
\emph{not} a Weierstrass point. Since $X$ is smooth at $P$, there exists a lift
of $P$ to a point $\widetilde{P} \in X(K[[t]])$ with
\begin{equation}
    \begin{split}
        x(\widetilde{P}) &= x(P)+t = x \\
        y(\widetilde{P}) &= y(P)+O(t).
    \end{split}
\end{equation}
We can think of $\widetilde{P}$ as expressing the expansion of the coordinates
$x,y$ with respect to the parameter $t$.  Indeed, we have
\begin{equation}
    y(\widetilde{P})=\sqrt{f(x(P)+t)} \in K[[t]]
\end{equation}
expanded in the usual way, since $f(x(P))\neq 0$ and the square root is
specified by $y(\widetilde{P})=y(P)+O(t)$.  Alternatively, we can think of
$\widetilde{P}$ as a formal neighborhood of $P$.

The Abel--Jacobi map, the putative endomorphism $\alpha$, and the Mumford map
extend to the ring $K[[t]]$.  By a lifting procedure, we will compute points
$\widetilde{Q}_1,\dots,\widetilde{Q}_g \in X(K[[t]])$ to arbitrary $t$-adic
precision such that
\begin{equation}
    \alpha_X\big(\widetilde{P}\big)
    =
    \big\{\widetilde{Q}_1,\dots,\widetilde{Q}_g\big\}
\end{equation}
with
\begin{equation}
    x\big(\widetilde{Q}_j\big) = x(Q_j) + O(t).
\end{equation}
We then attempt to fit a divisor $Y \subset X \times X$ defined over $K$ to the
point $\big\{(\widetilde{P},\widetilde{Q}_j)\}_{j}$, and proceed as before. The
only difference is that the divisor now interpolates this single infinitesimal
point instead of many points of the form $\big(R,\alpha_X(R)\big)$ with $R \in
X(\C)$.

\subsection{Lifting procedure}

For a generic choice of $P$, we may assume that $y(Q_j) \neq 0$ for all $j$ and
that the values $x(Q_j)$ are all distinct.  In practice, we may also keep $P$
and simply replace $\alpha \leftarrow \alpha+m$ with small $m \in \Z$ to
achieve this.

Let $x_j(t) \colonequals x(\widetilde{Q}_j)$.  The fact that the matrix
$M=(m_{ij})_{i,j}$ describes the action of $\alpha$ on the $F$-basis of
differentials $x^j \d x/y$ implies (by an argument described in detail in the
next section) that
\begin{equation}  \label{eqn:j1gxjM}
    \sum_{j=1}^g \frac{x_j^k \d x_j}{\sqrt{f(x_j)}}
    = \left( \sum_{j=0}^{g-1} m_{ij} x^j \right) \frac{\d x}{\sqrt{f(x)}}
\end{equation}
for all $k = 0, \dots, g-1$. In this equation, the branches of the square roots
are chosen so that $\sqrt{f(x)} = y(P) + O(t)$ and that $\sqrt{f(x_j)} = y(Q_j)
+ O(t)$ for all $j$.  Dividing by $\d x = \d t$, \eqref{eqn:j1gxjM} can be
rewritten in matrix form:
\begin{equation} \label{eqn:WD40}
    WD x' = \frac{1}{\sqrt{f(x)}} M w
\end{equation}
where
\begin{equation}
    \begin{aligned}
        W &\colonequals
        \begin{pmatrix}
            1 & \cdots & 1 \\
            x_1 & \cdots & x_g \\
            \vdots & \ddots & \vdots \\
            x_1^{g-1} & \cdots & x_g^{g-1}
        \end{pmatrix}, \\
        D &\colonequals \diag\big(\textstyle{\sqrt{f(x_1)}^{\,-1},\dots,
        \sqrt{f(x_g)}^{\, -1}}\big), \\
        x' &\colonequals \big(\!\d x_1/\d t, \ldots, \d x_g/\d t \big)^{\textsf{T}},
        \text{ and} \\
        w &\colonequals \big(1, x, \ldots,x^{g-1}\big)^{\textsf{T}},
    \end{aligned}
\end{equation}
where $\textsf{T}$ denotes the transpose. Since the values $x(Q_j) \in K$ are
all distinct, the Vandermonde matrix $W$ is invertible over $K[[t]]$.
Therefore, equation \eqref{eqn:WD40} allows us to solve for $x'$:
\begin{equation} \label{eqn:xpDW}
    x' = \frac{1}{\sqrt{f(x)}} D^{-1} W^{-1} M w.
\end{equation}
In practice, we use \eqref{eqn:xpDW} to solve for the series $x_j(t) \in
K[[t]]$ iteratively to any desired $t$-adic accuracy: if they are known up to
precision $O(t^n)$ for some $n \in \Z_{\geq 1}$, we may apply the identity
\eqref{eqn:xpDW} and integrate to get the series up to $O(t^{n+1})$.

\begin{example} \label{example:QMexample-Newton}
    We return to Example \ref{example:QMexample-IAJ}, and take $P=(2, 5)$ a
    non-Weierstrass point.
    We obtain
    \begin{equation}
        x_j(t) =
        \frac{1}{4} \bigl(3 \pm i \sqrt{3}\bigr)
        +
        \frac{1}{12} i \bigl(\sqrt{3} \pm 3 i\bigr) t
        +
        \frac{1}{144} \bigl(9 \mp 11 i \sqrt{3}\bigr) t^2
        +
        \frac{\pm 5 i}{36 \sqrt{3}} t^3
        + O\left(t^4\right)
        ,
    \end{equation}
    where $t=x - 2$ is a uniformizer at $P$.
    Taking advantage of the evident
    symmetry of $x_1, x_2$, we find
    \begin{equation} \label{eqn:x1x2t}
        x_1(t) + x_2(t) = \frac{4 t+6}{(t+2)^2}
        ,
        \quad
        x_1(t)x_2(t) = \frac{2 t+3}{(t+2)^2}.
    \end{equation}
    Thus
    \begin{equation}
        x_j(t) =
        \frac{2 t + 3 \pm i(t+1)\sqrt{2 t+3}}{(t+2)^2}.
    \end{equation}

    In Section~\ref{sec:fitdivisor} we will tackle the problem how to certify
    that $\alpha$ is indeed an endomorphism, and that the rational functions
    \eqref{eqn:x1x2t} are correct: see Example \ref{exm:certrunexample}.

    Here is another way: for genus 2 curves we have an upper bound for the
    degrees of $x_1(t) + x_2(t)$ and $x_1(t)x_2(t)$ as rational functions,
    given by
    \begin{equation}
      d: = \tr(\alpha \alpha^\dagger) = \tr(R J R^{\textsf{T}} J^{-1}) / 2 =
        \langle \alpha(\Theta), \Theta \rangle,
    \end{equation}
    where $\dagger$ denotes the Rosati involution and $J$ is the standard
    symplectic matrix; see van Wamelen \cite[\S 3]{vanwamelen-cm} for more
    details and Remark~\ref{rmk:daggerpi2} for a possible generalization to
    higher genus. Therefore,  to deduce the pair $(x_1(t) + x_2(t),
    x_1(t)x_2(t))$ it is sufficient to compute $x_j(t)$ up to precision $O(t^{2
    d + 1})$. Furthermore, we may sped up the process significantly by doing
    this modulo many small primes and applying a version of the Chinese
    remainder theorem with denominators (involving LLL).

    In this example, we have $d = 4$ and we deduced the pair $(x_1(t) + x_2(t),
    x_1(t)x_2(t))$ modulo 131 62-bit primes that split completely in
    $\Q(\sqrt{2}, \sqrt{-3})$ by computing $x_j (t)$ up to precision
    $O(t^{9})$. All together deducing $(x_1(t) + x_2(t), x_1(t)x_2(t))$ given
    $(Q_1, Q_2)$ took less than 5 CPU seconds on a standard desktop machine.

    (A third possible way to certify $\alpha$ using \eqref{eqn:x1x2t} is by
    following van Wamelen's approach \cite[\S 9]{vanwamelen-cm}.)
\end{example}

\section{Puiseux lift} \label{sec:puiseux}

In the previous section, we lifted a single computation of $\alpha_X(P) =
\sum_{j=1}^g Q_j - O_0$ to a formal neighborhood.  In this section, we show how
one can dispense with even this one numerical computation to obtain an exact
certification algorithm for the matrix of a putative endomorphism.

\subsection{Setup}

We continue our notation but restore generality, once more allowing $X$ to be a
general curve.  We may for example represent $X$ by a plane model that is
smooth at $P_0$ (but possibly with singularities elsewhere).  Let $P_0 \in
X(K)$ and let $M \in \M_g(K)$ be the tangent representation of a putative
endomorphism $\alpha$ on an $F$-basis of $H^0 (X, \omega_X)^*$.

We now make the additional assumption that $P_0$ is \emph{not} a Weierstrass
point. Then by Riemann--Roch, the map
\begin{align}\label{eq:iso_at_P0}
    \begin{split}
        \Sym^g (X) & \to J \\
        \{Q_1, \dots , Q_g\} & \mapsto \sum_{j=1}^g (Q_j - P_0)
    \end{split}
\end{align}
is locally an isomorphism around $\{P_0, \dots , P_0\}$, in the sense that it
is a birational map that restricts to an isomorphism in a neighborhood of said
point.

Let $x \in F(X)$ be a local parameter for $X$ at $P_0$.  Then $x \colon X \to
\PP^1$ is also a rational function, and we use the same symbol for this map.
Since $X$ is smooth at $P_0$, we obtain a canonical point $\widetilde{P}_0 \in
X (F [[x]])$ such that:
\begin{enumerate}
    \item[(i)] $\widetilde{P}_0$ reduces to $P_0$ under the reduction map
        $X(F[[x]]) \to X(F)$, and
    \item[(ii)] $x (\widetilde{P}_0) = x \in F[[x]]$.
\end{enumerate}
On an affine open set $U \ni P_0$ of $X$ with $U$ embedded into affine space
over $F$, we may think of $\widetilde{P}_0$ as providing the local expansions
of the coordinates at $P_0$ in the local ring at $P_0$.

Since \eqref{eq:iso_at_P0} is locally an isomorphism at $P_0$, we can locally
describe $\alpha_X(\widetilde{P}_0)$ uniquely as
\begin{equation}
    \alpha_X(\widetilde{P}_0)
    =
    \{\widetilde{Q}_1,\dots,\widetilde{Q}_g\} \in \Sym^g(X)(F[[x]]) .
\end{equation}
The reduction to $F$ of $\{\widetilde{Q}_i\}_i$ is the $g$-fold multiple
$\{P_0,\dots,P_0\} \in \Sym^g(X)(F)$. The map $X^g \to \Sym^g (X)$ is ramified
above $\{P_0, \dots , P_0\}$, so in general we cannot expect to have
$\widetilde{Q}_i \in X(F[[x]])$. Instead, consider the generic fiber of the
point $\{\widetilde{Q}_i\}_i$, an element of $\Sym^g(X)(F((x)))$; this generic
fiber lifts to a point of $X^g$ defined over some finite extension of $F
((x))$.  Since $\opchar F =0$, the algebraic closure of $F((x))$ is the field
$F^{\alg}((x^{1/\infty}))$ of Puiseux series over $F^{\alg}$. Since $X$ is
smooth at $P_0$, the lift of $\{\widetilde{Q}_i\}_i$ is even a point on $X^g$
over the ring of integral Puiseux series $F^{\alg}[[x^{1/\infty}]]$.

In other words, if we allow ramification (fractional exponents) in our formal
expansion, we can deform the equality $\alpha_X(P_0)=\{P_0,\dots,P_0\}$ to a
formal neighborhood of $P_0$.

\subsection{Lifting procedure}

The lifting procedure to obtain this deformation algorithmically is similar to
the one outlined in the previous section; here we provide complete details. For
$i=1,\dots,g$, let
\begin{equation}
    \omega_i = f_i \d x
\end{equation}
be an $F$-basis of $H^0(X,\omega_X)$ with $f_i \in F(X)$.  The functions $f_i$
are by definition regular at $P_0$, so they admit a power series expansion
$f_i(x) \in F[[x]]$ in the uniformizing parameter $x$.  Because $P_0$ is not a
Weierstrass point, we may without loss of generality choose $\omega_i$ in row
echelonized form, i.e., so that
\begin{equation} \label{eqn:omegai}
    \omega_i = ( x^{i-1} + O(x^i) ) \d x
\end{equation}
for $i=1,\dots,g$.  (If it is more convenient, we may even work with a full
echelonized basis.)

For $j=1,\dots,g$, let
\begin{equation}
    x_j = x(\widetilde{Q}_j) \in F^{\alg}[[x^{1/\infty}]]
\end{equation}
be the $x$-coordinates of the points $\widetilde{Q}_j$ on the graph of $\alpha$
above $\widetilde{P}$.

\begin{proposition}\label{prop:pullpush}
    Let $\left\{ \omega_1, \dots , \omega_g \right\}$ be a basis of $H^0 (X,
    \omega_X)$, with $\omega_i = f_i \d x$ around $P_0$.
    Let $M = (m_{i,j})_{i,j}$ be the tangent representation of $\alpha$ with
    respect to the dual of this basis. Then we have
    \begin{equation}\label{eq:lift}
        \sum_{j=1}^g f_i (x_j) \d x_j
        =
        \sum_{j=1}^g m_{i,j} f_j (x) \d x
        \quad \text{ for all $i = 1, \dots , g$.}
    \end{equation}
\end{proposition}

\begin{proof}
    This is essentially proven by Smith \cite[\S3.5]{smith-thesis}. Let $Y$ be
    the divisor corresponding to $\alpha$, and let $\pi_1$ and $\pi_2$ be the two
    projection maps from $Y$ to $X$. Then $\alpha^* = (\pi_2)_* \pi_1^*$ (see
    \emph{loc.\ cit.}), which in an infinitesimal neighborhood of $P_0$ becomes
    \eqref{eq:lift}.

    An alternative argument is as follows. By construction, we have
    \begin{equation}\label{eq:pullpush1}
        \sum_{j = 1}^g (\widetilde{Q}_j - P_0) = \alpha (\widetilde{P}_0 - P_0) .
    \end{equation}
    On the tangent space, addition on the Jacobian induces the usual addition.
    Considering both sides of \eqref{eq:pullpush1} over
    $F^{\alg}[[x^{1/\infty}]]$ and substituting the resulting power series in the
    differential form $\omega_i$, we obtain
    \begin{equation}
        \sum_{j=1}^g x_j^* (\omega_i)
        =
        x^* (\alpha^* (\omega_i)))
        \quad \text{ for all $i = 1, \dots , g$,}
    \end{equation}
    which also yields \eqref{eq:lift}.
\end{proof}

We iteratively solve \eqref{eq:lift} as follows.  We begin by computing initial
expansions
\begin{equation}
    x_j = c_{j,\nu} x^{\nu} + O(x^{\nu+1/e})
\end{equation}
where
\begin{equation} \label{eqn:computenu}
    \nu \colonequals \min_{i,j}(\{j/i : m_{i, j} \neq 0\}) \in \Q_{>0},
\end{equation}
and where $e$ is the denominator of $\nu$. Note that $\nu$ is well-defined
since the matrix $M$ has full rank; typically, but not always, we have
$\nu=1/g$. Combining the notation above with \eqref{eqn:omegai} we obtain
\begin{equation}
    \begin{aligned}
        x f_i(x_j) \d x_j
        &= ((c_{j,\nu}x^{\nu})^{i-1} + O(x^{i\nu}))\left(\nu c_{j,\nu}x^{\nu}
    + O(x^{\nu+1/e})\right) \d x \\
    &= (\nu c_{j,\nu}^ix^{i\nu} + O(x^{i\nu+1/e}))\d x.
\end{aligned}
\end{equation}
Inspecting the leading terms of \eqref{eq:lift} for each $i$ we obtain
\begin{equation}
    \sum_{j = 1} ^g (\nu c_{j,\nu}^i x^{i\nu} + O(x^{i\nu + 1/e}))\d x
    =
    \sum_{j=1}^g m_{i,j}(x^j + O(x^{j+1}))\d x ,
\end{equation}
therefore for all $i$ we have
\begin{equation} \label{eqn:defnu}
    \nu \sum_{j=1}^g c_{j,\nu}^i = m_{i,i\nu},
\end{equation}
where $m_{i,i\nu}=0$ if $i\nu \not\in \Z$. The equations \eqref{eqn:defnu} are
symmetric under the action of the permutation group $S_g$, and up to this
action there is a unique nonzero solution by Newton's formulas, as $m_{i,i\nu}
\neq 0 $ for some $i$.

The equations \eqref{eqn:defnu} are of different degree with respect to the
leading terms $c_{j,\nu}$.  Therefore, replacing $\alpha$ by $\alpha+m$ with $m
\in \Z$ will eventually result in a solution with distinct $c_{j,\nu}$.  For
purposes of rigorous verification it is the same to verify $\alpha$ as it is
$\alpha+m$, so we may suppose that the values $c_{j,\nu}$ are distinct.

Having determined the expansions
\begin{equation}
    x_j = c_{j,\nu}x^{\nu} + c_{j,\nu+1/e}x^{\nu+1/e} + \dots
    + c_{j,\nu+n/e}x^{\nu+n/e} + O(x^{\nu+(n+1)/e})
\end{equation}
for $j=1,\dots,g$ up to some precision $n \geq 1$, we integrate \eqref{eq:lift}
to iteratively solve for the next term in precision $n+1$. As at the end of the
previous section, we then introduce new variables $c_{j,\nu+(n+1)/e}$ for the
next term and consider the first coefficients on the left hand side of the
equations \eqref{eq:lift} in which these new variables occur. Because of our
echelonization and the presence of the derivative $\d x_j$, the exponents of $x$
for which these coefficients occur are
\begin{equation}
    \nu - 1 + (n+1)/e, 2 \nu - 1 + (n+1)/e, \dots, g \nu - 1 + (n+1)/e.
\end{equation}
We obtain an inhomogeneous linear system in the new variables whose homogeneous
part is described by a Vandermonde matrix in $c_{1,\nu}, \dots, c_{g,\nu}$.
This system has a unique solution since we have ensured that the latter
coefficients are distinct. The Puiseux series $x_j = x(\widetilde{Q}_j)$ for
each $j$ then determines the point $\widetilde{Q}_j$ because we assumed $x$ to
be a uniformizing element.

\begin{remark}
    In practice, we iterate the approximations $x_j$ by successive Hensel
    lifting. Indeed, let $F_i$ be the formal integral of the function $f_i$,
    and let $F$ be the multivariate function $(F_1, \dots, F_g)$. Then the
    equation \eqref{eq:lift} is equivalent to solving for $x_1, \dots, x_g$ in
    \begin{equation}
        F (x_1, \dots , x_g)
        =
        \left(\sum_{j = 1}^g m_{1,j} F_i (x), \dots,
        \sum_{j = 1}^g m_{g,j} F_g (x)\right).
    \end{equation}
    Our initialization is a sufficiently close approximation for the Hensel
    lifting process to take off.
\end{remark}

\begin{example}
    We compute Example \ref{example:QMexample-Newton} again, but starting afresh
    with just the matrix $M = \begin{pmatrix} 0 & \sqrt{2} \\ \sqrt{2} & 0
    \end{pmatrix}$ and the point $P_0 = (0,\sqrt{-1})$.  In order to be able to
    display our results, we work modulo a prime above $4001$ in
    $K = \Q(\sqrt{-1},\sqrt{2})$. We first expand
    \begin{equation}
        \widetilde{P}_0 = (x, 3102+247x+1714x^2+2082x^3+1505x^4+O(x^5)).
    \end{equation}
    By \eqref{eqn:computenu}, we have $\nu=1/2$.  The equations \eqref{eqn:defnu}
    read:
    \begin{equation}
        \begin{aligned}
            c_{1,1/2}+c_{2,1/2} &= 2m_{1,1/2}=0 \\
            c_{1,1/2}^2+c_{2,1/2}^2 &= 2m_{2,1} = 2\sqrt{2}
        \end{aligned}
    \end{equation}
    so $c_{2,1/2}=-c_{1,1/2}$ and $c_{1,1/2}^2 = \sqrt{2}$, giving
    \begin{equation}
        c_{1,1/2} \equiv 2559 \psmod{4001},
        \quad c_{2,1/2} \equiv -2559 \equiv 1442 \psmod{4001}.
    \end{equation}
    Now iteratively solving the differential system \eqref{eq:lift}, we find
    \begin{equation}
        \begin{aligned}
            \widetilde{Q_1} &= (2559x^{1/2} + 1445x + 2635x^{3/2} + O(x^2), \\
                            &\qquad 3102 + 3916x^{1/2} + 3938x + 1271x^{3/2} + O(x^2)) \\
            \widetilde{Q_2} &= (1442x^{1/2} + 1445x + 1366x^{3/2} + O(x^2), \\
                            &\qquad 3102 + 85x^{1/2} + 3938x + 2730x^{3/2} + O(x^2)).
        \end{aligned}
    \end{equation}
    We use these functions directly to interpolate a divisor in
    the next section (and we also consider the Cantor representation, involving
    in particular their symmetric functions).
\end{example}

\section{Proving correctness}
\label{sec:fitdivisor}

The procedures described in the previous sections work unimpeded for any matrix
$M$, including those that do \emph{not} correspond to actual endomorphisms. In
order for $M$ to represent an honest endomorphism $\alpha \in \End(J_K)$, we
now need to fit a divisor $Y \subset X \times X$ representing the graph of
$\alpha$.

\subsection{Fitting and verifying}

We now proceed to fit a divisor to either the points computed numerically or
the Taylor or Puiseux series in a formal neighborhood computed exactly.  The
case of numerical interpolation was considered by Kumar--Mukamel
\cite{kumar-mukamel}, and the case of Taylor series is similar, so up until
Proposition~\ref{prop:verify} below we focus on our infinitesimal versions.

Let $\pi_1, \pi_2 \colon X \times X \to X$ be the two projection maps. If the
matrix $M$ corresponds to an endomorphism, then the divisor $Y$ traced out by
the points $(\widetilde{P}_0,\widetilde{Q}_j)$ has degree $g$ with respect to
$\pi_1$ and degree $d$ with respect to $\pi_2$ for some $d \in \Z_{\geq 1}$.
Accordingly, we seek equations defining this divisor.

Choose an affine open $U \subset X$, with a fixed embedding into some ambient
affine space. We then try to describe $D \subset U \times U$ by choosing degree
bounds $n_1,n_2 \in \Z_{\geq 1}$ (with usually $n_2 = g$) and considering the
$K$-vector space $K[U \times U]_{\leq(n_1, n_2)}$ of regular functions on $U
\times U$ that are of degree at most $n_1$ when considered as functions on $U
\times \left\{ P_0 \right\}$ and degree at most $n_2$ on $\left\{ P_0 \right\}
\times U$. Let $N \colonequals \dim_K K[U \times U]_{\leq (n_1, n_2)}$ be the
dimension of this space of functions.  We then develop the points
$(\widetilde{P}_0, \widetilde{Q}_j)$ to precision $N + m$ for some suitable
global margin $m \geq 2$, and compute the subspace $Z \subseteq K[U \times
U]_{\leq (n_1 , n_2)}$ of functions that annihilates all these points to the
given precision. If $M$ is the representation of an actual endomorphism, then
we will in this way eventually find equations satisfied by $(\widetilde{P}_0,
\widetilde{Q}_j)$ by increasing $n_1,n_2$.

We show how to verify that a putative set of such equations is in fact correct.
Assume that $Z$ contains a nonzero function (on $U \times U$), and let $E$ be
the subscheme of $X \times X$ defined by the vanishing of $Z$.

\begin{proposition}\label{prop:verify}
    Suppose that the second projection $\pi_2$ maps $E$ surjectively onto $X$
    and that the intersection of $E$ with $\left\{ P_0 \right\} \times X$
    consists of a single point with multiplicity $g$. Then $M$ defines an
    endomorphism of $J_K$.
\end{proposition}

\begin{proof}
    We have ensured that a nonzero function on $U \times U$ vanishes at $E$, so
    $E$ cannot be all of $X \times X$. Yet the subscheme $E$ cannot be of
    (Krull) dimension $0$ either because $E$ surjects to $X$. Therefore $E$ is
    of dimension $1$.

    Let $Y \subset E$ be the union of the irreducible components of dimension
    $1$ of $E$ that contain the points $(\widetilde{P}_0, \widetilde{Q}_j)$.
    Because the degree of the projections to the second factor do not depend on
    the chosen base point, our hypothesis on the intersection of $E$ with
    $\left\{ P_0 \right\} \times X$ ensures that $E \smallsetminus Y$ consists
    of a union of points and vertical divisors: these define the trivial
    endomorphism.

    The subscheme $Y \subset X \times X$ defines a (Weil or Cartier) divisor
    whose projection to the second component is of degree $g$, and such a
    divisor defines an endomorphism \cite[\S 3.5]{smith-thesis}. The fact that
    $Y$ contains the points $(\widetilde{P}_0, \widetilde{Q}_j)$, which we
    chose to satisfy \eqref{eq:lift} over $K$ with a suitable nontrivial
    margin, then ensures without any further verification the endomorphism
    enduced by $Y$ has tangent representation $M$.
\end{proof}

The hypotheses of Proposition \ref{prop:verify} can be verified
algorithmically, for example by using Gr\"obner bases. Indeed, the property
that $\pi_2$ maps $E$ surjectively to $X$ can be verified by calculating a
suitable elimination ideal, and the degree of the intersection with $\left\{
P_0 \right\} \times X$ is the dimension over $K$ of the space of global
sections of a zero-dimensional scheme. If desired, the construction of the
divisor $Y$ from $E$ is also effectively computable, calculating irreducible
components via primary decomposition.

In a day-and-night algorithm, we would alternate the step of seeking to fit a
divisor (running through an enumeration of the possible values $(n_1, n_2)$
above) with refining the numerical endomorphism ring by computing with
increased precision of the period matrix.  If $M$ does not correspond to an
endomorphism, then we will discover this in the numerical computation (provably
so, if one works with interval arithmetic to keep track of errors in the
numerical integration). On the other hand, if $M$ does correspond to an
endomorphism, then eventually a divisor will be found, since increasing $n_1$
and $n_2$ eventually yields generators of the defining ideal of the divisor in
$U \times U$ defined by $M$, which we can prove to be correct by using
Proposition \ref{prop:verify}. Therefore we have a deterministic algorithm that
takes a putative endomorphism represented by a matrix $M \in \M_g(F^{\alg})$
and returns \textsf{true} or \textsf{false} according to whether or not $M$
represents an endomorphism of the Jacobian.

\begin{remark}
    More sophisticated versions of the approach above are possible, for example
    by using products of Riemann--Roch spaces instead of using the square of
    the given ambient space. Additionally, the algorithm can be significantly
    sped up by determining the divisor $Y$ modulo many small primes and
    applying a version of the Chinese remainder theorem with denominators
    (involving LLL) to recover the defining ideal of $Y$ from its reductions.
\end{remark}

\begin{remark}
    Conversely, if we have a divisor $Y \subset X \times X$ (not necessarily
    obtained from the Taylor or Puiseux method), we can compute the tangent
    representation of the corresponding endomorphism as follows. Choose a point
    $P_0$ on $X$ such that the intersection of $\{P_0\} \times X$ with $Y$ is
    proper, with
    \begin{equation}
        Y \cap (\{P_0\} \times X) =\left\{ Q_1, \dots , Q_e \right\}
    \end{equation}
    the points $Q_e$ taken with multiplicity.  Then we can again develop the
    points $Q_j$ infinitesimally, and as long as \eqref{eq:lift} is verified
    for the initial terms, the divisor $Y$ induces an endomorphism with $M$ as
    tangent representation.
\end{remark}

\begin{remark} \label{rmk:daggerpi2}
    While the above method will terminate as long as $M$ corresponds to an
    actual endomorphism, Khuri-Makdisi has indicated an upper bound $D$ of the
    degree of $\pi_2$ to us, namely $(g - 1)! \tr (\alpha \alpha^{\dagger})$,
    where $\dagger$ denotes the Rosati involution. Such an upper bound would
    allow us to rule out a putative tangent matrix $M$ as one \emph{not}
    corresponding to an endomorphism without resort to a numerical computation.

    Indeed, having calculated the upper bound $D$, we can take $(n_1, n_2) =
    (D, g)$ above, and a suitably large $N$ can be determined by applying a
    version of the Riemann-Roch theorem for surfaces. After determining the
    resulting equations, Proposition \ref{prop:verify} can be used to tell us
    conclusively whether we actually obtain a suitable divisor or not. However,
    since our day-and-night algorithm is provably correct and functions very
    well in practice, we have not elaborated these details or implemented this
    approach.
\end{remark}

\begin{example} \label{exm:certrunexample}
    We revisit our running example one last time. Recall that
    \begin{equation}
        X \colon y^2 = x^{5} - x^{4} + 4 x^{3} - 8 x^{2} + 5 x - 1
    \end{equation}
    and
    \begin{equation}
        M = \left(\begin{array}{rr}
                0 & \sqrt{2} \\
                \sqrt{2} & 0
        \end{array}\right) .
    \end{equation}
    While $X$ may not have an obvious Weierstrass point, we can apply a trick
    that is useful for general hyperelliptic curves. Instead of $X$, we
    consider the quadratic twist of $X$ by $-1$, namely
    \begin{equation}
        X' \colon y^2 = -(x^{5} - x^{4} + 4 x^{3} - 8 x^{2} + 5 x - 1) ,
    \end{equation}
    which has the rational non-Weierstrass point $P_0 = (0, 1)$. While the
    curves $X$ and $X'$ are not isomorphic, their endomorphism rings are,
    because the isomorphism $(x, y) \to (x, \sqrt{-1} y)$ induces a scalar
    multiplication on global differentials, which disappears when changing
    basis by it.

    We find a divisor with $d = 4$ with respect to $\pi_2$ (matching
    Khuri-Makdisi's estimate $(g - 1)! \tr (\alpha \alpha^{\dagger}) = 4$ from
    Remark \ref{rmk:daggerpi2}).  Using a margin $m = 16$, the number of terms
    needed in the Puiseux expansion to find enough equations of $Y$ equals
    $48$. On a standard desktop machine, this calculation took less than $3$
    CPU seconds.

    The equations defining the divisor $Y$ representing $M$ are quite long and
    unpleasant, so that we cannot reproduce them here. As mentioned in the
    introduction, they are available in the repository that contains our
    implementation. However, we can indicate the induced divisor mapped to
    $\P^1 \times \P^1$ under the hyperelliptic involution: it is given by
    \begin{small}
        \begin{equation}
            \begin{split}
                & 4 x_2^4 x_1^8 + 4 x_2^4 x_1^7 + (-96 \sqrt{2} + 29) x_2^4 x_1^6
                + 2 (48 \sqrt{2} - 9) x_2^4 x_1^5 + (-312 \sqrt{2} + 1193) x_2^4 x_1^4 \\ &
                + 4 (216 \sqrt{2} - 891) x_2^4 x_1^3 + 4 (-210 \sqrt{2} + 959) x_2^4 x_1^2
                + 4 (84 \sqrt{2} - 440) x_2^4 x_1 + 4 (-12 \sqrt{2} + 73) x_2^4 \\ &
                + 4 (2 \sqrt{2} + 2) x_2^3 x_1^8 + 4 (-39 \sqrt{2} - 65) x_2^3 x_1^7
                + 4 (107 \sqrt{2} + 597) x_2^3 x_1^6 + 4 (-120 \sqrt{2} - 1864) x_2^3 x_1^5 \\ &
                + 4 (152 \sqrt{2} + 2649) x_2^3 x_1^4 + 4 (-243 \sqrt{2} - 1945) x_2^3 x_1^3
                + 4 (223 \sqrt{2} + 776) x_2^3 x_1^2 + 4 (-84 \sqrt{2} - 166) x_2^3 x_1 \\ &
                + 4 (10 \sqrt{2} + 24) x_2^3 + 4 (-2 \sqrt{2} + 2) x_2^2 x_1^8
                + 2 (164 \sqrt{2} + 51) x_2^2 x_1^7 + 2 (-664 \sqrt{2} - 1543) x_2^2 x_1^6 \\  &
                + 4 (340 \sqrt{2} + 3770) x_2^2 x_1^5 + 2 (-348 \sqrt{2} - 13363) x_2^2 x_1^4
                + 2 (484 \sqrt{2} + 10499) x_2^2 x_1^3 \\ &
                + 4 (-196 \sqrt{2} - 1841) x_2^2 x_1^2 + 4 (20 \sqrt{2} + 301) x_2^2 x_1
                + 4 (12 \sqrt{2} - 46) x_2^2 + 4 (-5 \sqrt{2} - 9) x_2 x_1^8 \\ &
                + 4 (-24 \sqrt{2} + 12) x_2 x_1^7 + 4 (358 \sqrt{2} + 226) x_2 x_1^6
                + 4 (-303 \sqrt{2} - 2210) x_2 x_1^5 + 4 (63 \sqrt{2} + 4242) x_2 x_1^4 \\ &
                + 4 (-508 \sqrt{2} - 2960) x_2 x_1^3 + 4 (538 \sqrt{2} + 623) x_2 x_1^2
                + 4 (-139 \sqrt{2} + 40) x_2 x_1 + (8 \sqrt{2} + 33) x_1^8 \\ &
                + 4 (-2 \sqrt{2} + 4) x_1^7 + 4 (-106 \sqrt{2} + 19) x_1^6
                + 2 (164 \sqrt{2} + 807) x_1^5 - 3348 x_1^4 + 4 (166 \sqrt{2} + 515) x_1^3 \\ &
                + (-720 \sqrt{2} - 223) x_1^2 + 4 (46 \sqrt{2} - 24) x_1 = 0 .
            \end{split}
        \end{equation}
    \end{small}
\end{example}

\subsection{Cantor representation}

In certain situations it might be more convenient directly to compute the
rational map
\begin{equation}\label{eq:maptosym}
    \alpha_X \colon X \dashrightarrow \Sym^g (X).
\end{equation}
This can be done as follows. Choose an affine model of $f (x, y) = 0$ for $X$.
Then a generic divisor of degree $g$ on $X$ can be described by equations of
the form
\begin{equation}\label{eq:cantorrep}
    \begin{split}
        x^g + a_1 x^{g - 1} + \dots + a_{g - 1} x + a_g = 0 , \\
        y = b_1 x^{g - 1} + \dots + b_{g - 1} x + b_g ,
    \end{split}
\end{equation}
which we call a \defi{Cantor representation}. Using $f$ one can determine $g$
equations in the $a_i$ and $b_i$ that conversely determine when a generic point
of the form \eqref{eq:cantorrep} defines a divisor of degree $g$ on $X$.

After fixing our origin in some point $P_0$ as before, \eqref{eq:cantorrep}
also gives a description of generic divisors of degree $0$ on $X$. By taking a
sufficiently precise development $(\widetilde{P}_0, \widetilde{Q}_j)$, we can
obtain $a_i$ and $b_i$ as functions in $K (X)$, increasing this precision as we
try functions of larger degree. In the end, we can verify these rational
functions by checking that the equations \eqref{eq:cantorrep} are satisfied and
additionally checking that the corresponding tangent representation is correct.
As above, we see that for this final step it suffices to check that the initial
terms of the Puiseux approximation cancel \eqref{eq:cantorrep}.

\subsection{Splitting the Jacobian} \label{sec:splitjac}

The algorithms above can be generalized to the verification of the existence of
homomorphisms $\Jac (X) \to \Jac (Y)$, which can be represented by either a
rational map $X \dashrightarrow \Sym^{g_Y} (Y)$ or a divisor on $X \times Y$.
In particular, this allows us to verify factors of the Jacobian variety that
correspond to curves, as explained by Lombardo \cite[\S 6.2]{lombardo-endos} in
genus $2$. For curves of genus $3$, we can similarly identify curves of genus
$2$ that arise in their Jacobian, by reconstructing these genus $2$ curves from
their period matrices after choosing a suitable polarization.

\subsection{Saturation} \label{sec:saturation}

The methods above allow us to certify that the tangent representation $M \in
\M_g(K)$ of a putative endomorphism is correct. If we are also given that the
period matrix $\Pi$ is correct up to some (typically small) precision---for
hyperelliptic curves, one may use Molin's double exponentiation algorithm
\cite[Th\'{e}or\`{e}me 4.3]{molin}---we may also deduce that the geometric
representation $R \in \M_{2g}(\Z)$ in \eqref{equation:analytic_geometric} is
also correct. Assuming that we have verified the geometric representation of
all the generators of the endomorphism algebra, we can then also recover the
endomorphism ring by considering possible superorders and ruling them out.

\begin{example}
    For example, take $X\colon y^2 = -3 x^6+8 x^5-30 x^4+50 x^3-71 x^2+50 x-27$
    to be a simplified Weierstrass model for the genus 2 curve with LMFDB label
    \href{http://www.lmfdb.org/Genus2Curve/Q/961/a/961/2}{\texttt{961.a.961.2}}.
    We can then verify that the endomorphism algebra is $\Q(\sqrt{5})$, and
    $\sqrt{5}$ is represented by
    \begin{equation}
        M =
        \begin{pmatrix}
            -1 & 2 \\
            2 & 1
        \end{pmatrix}
        \quad \text{and} \quad
        R =
        \begin{pmatrix}
            -1 & 0 & 0 & -1
            \\
            1 & 1 & 1 & 0
            \\
            0 & 4 & -1 & 1
            \\
            -4 & 0 & 0 & 1
        \end{pmatrix}.
    \end{equation}
    From the above computation, we also deduce that the endomorphism ring is
    $\Z[\sqrt{5}]$ and not the superorder $\Z[(1 + \sqrt{5})/2]$, as $1 + R
    \notin 2M_{4}(\Z)$.
\end{example}

\section{Upper bounds}
\label{sec:upperbounds}

In this section, we show how determining Frobenius action on $X$ for a large
set of primes often quickly leads to sharp upper bounds on the dimension of the
endomorphism algebra of the Jacobian $J$ of $X$.

\subsection[Genus 2 upper bounds]{Upper bounds in genus 2 via N\'eron--Severi rank} \label{sec:71g2}

We begin with upper bounds for curves of genus 2.  Lombardo \cite[\S
6]{lombardo-endos} has already given a practical method for these curves; we
consider a slightly different approach.

Suppose $X$ has genus $2$. Then its Jacobian $J$ is naturally a principally
polarized abelian \emph{surface}; let $\dagger$ denote its Rosati involution.
In this case, we can take advantage of the relation between the N\'eron--Severi
group $\NS(J)$ and $\End(J)_\Q$: by Mumford \cite[Section 21]{mumford}, we have
an isomorphism of $\Q$-vector spaces
\begin{equation} \label{eqn:rosati}
    \NS ( J )_\Q \simeq \{ \phi \in \End( J )_\Q : \phi^{\dagger} = \phi \} .
\end{equation}

Let $\rho( J ) \colonequals \rk \NS( J )$.  By Albert's classification of
endomorphism algebras,
\begin{equation} \label{eqn:rhotoEnd}
    \rho(J^{\alg}) =
    \begin{cases}
        4, & \text{if } \End(J^{\alg})_\R \simeq \M_2(\C); \\
        3, & \text{if } \End(J^{\alg})_\R \simeq \M_2(\R); \\
        2, & \text{if } \End(J^{\alg})_\R \simeq \R\times\R, \C\times\C \text{ or } \C \times \R; \\
        1, & \text{if } \End(J^{\alg})_\R \simeq \R.
    \end{cases}
\end{equation}
So if we had a way to compute $\rho(J^{\alg})$, we could limit the number of
possibilities for $\End(J^{\alg})_\R$, and hit it exactly in many cases
including the typical case when $\End(J^{\alg})=\Z$.  To compute
$\rho(J^{\alg})$, we look modulo primes.

Let $\frakp$ be a nonzero prime of (the ring of integers of) $F$ with residue
field $\F_\frakp$.  Let $\F_\frakp^{\alg}$ be an algebraic closure of $\Fp$.
Suppose that $X$ has good reduction $X_{\Fp}$ at $\frakp$. We write
$J_\frakp=J_{\Fp}$ for the reduction of $J$ modulo $\frakp$ and
$J_\frakp^{\alg}=J_{\F_\frakp^{\alg}}$ its base change to $\F_\frakp^{\alg}$.
(There is no ambiguity in this notation, as $(J^{\alg})_\frakp$ does not make
sense.) Then there is a natural injective specialization homomorphism of
$\Z$-lattices
\begin{equation} \label{eqn:NsJfalg}
    s_p\colon \NS (J^{\alg}) \hookrightarrow \NS ( J_{\frakp}^{\alg} ),
\end{equation}
so $\rho(J^{\alg}) \leq \rho(J_{\frakp}^{\alg})$.

Let $q=\#\Fp$, let $\Frob_\frakp$ be the $q$-power Frobenius automorphism, and
let $\ell \nmid q$ be prime.  Let
\begin{equation}
    \begin{aligned}
        c_\frakp(T) \colonequals &
        \det\left( 1 - \Frob_{\frakp} T \,|\, H_{\et}^1( J^{\alg}, \Q_{\ell}) \right)\\
        = & \det\left( 1 - \Frob_{\frakp} T \,|\, H_{\et}^1( X^{\alg}, \Q_{\ell}) \right) \\
        = & 1 + a_1 T + a_2 T^2 + a_1 q T^3 + q^2 T^4 \in 1+T\Z[T].
    \end{aligned}
\end{equation}
Then
\begin{equation}
    \begin{aligned}
        c_\frakp^{\wedge 2}(T) \colonequals
        & \det\left( 1 - \Frob_{\frakp} T \,|\, H_{\et}^2( J^{\alg}, \Q_{\ell}) \right) \\
        =& \det\left( 1 - \Frob_{\frakp} T \,|\, \textstyle{\bigwedge}^2 H_{\et}^1( J^{\alg}, \Q_{\ell}) \right)\\
        = & (1 - q T)^2 (1 + (2 q - a_2)T + (2 q + a_1 ^2 - 2 a_2) q T^2
        + (2 q - a_2) q^2 T^3 + q^4 T^4).
    \end{aligned}
\end{equation}

The Tate conjecture holds for abelian varieties over finite fields \cite{tate},
and it relates $\NS(J_\Fp)$, as a lattice with its intersection form, with
$c_\frakp^{\wedge 2}(T)$ in the following way.

\begin{proposition} \label{prop:tate}
    The following statements hold.
    \begin{enumalph}
    \item $\rho(J_{\frakp}^{\alg})$ is equal to the number of reciprocal roots
        of $c_\frakp^{\wedge 2}(T)$ of the form $q$ times a root of unity.
    \item We have
        \begin{equation} \label{disc}
            \disc( \NS(J_{\frakp}) )
            = \lim_{s \rightarrow 1} \frac{  (-1)^{\rho(J_{\frakp}) -1}
            c_\frakp^{\wedge 2}(q^{-s})}{ q (1-q^{1-s})^{\rho(J_{\frakp})} }
            \bmod{\Q^{\times 2}}.
        \end{equation}
    \end{enumalph}
\end{proposition}

\begin{proof}
    For part (a), we know that $\rho(X_{\frakp})$ is equal to the multiplicity
    of $q$ as a reciprocal root of $c_\frakp^{\wedge 2}(T)$ by the Tate
    conjecture, and (a) follows by taking a power of the Frobenius.  For part
    (b), the Tate conjecture implies the Artin--Tate conjecture by work of
    Milne \cite[Theorem 6.1,][]{milne-tate,milne-tate-a}, which implies (b)
    after simplification using that $\# \Br(X)$ is a perfect square
    \cite{liu-lorenzini-raynaud}.
\end{proof}

We will use one other ingredient: we can rule out the possibility that
$J^{\alg}$ has CM by looking at $c_\frakp(T)$ as follows.

\begin{lemma} \label{lem:endJFalg}
    Suppose that $\End(J^{\alg})_\Q=L$ is a quartic CM field.  Let $\frakp$ be
    a prime of $F$ of good reduction for $X$, let $p$ be the prime of $\Q$
    below $\frakp$, and suppose that $p$ splits completely in $L$.  Then
    $c_\frakp(T)$ is irreducible and
    \begin{equation}  \label{eqn:QTCM}
        L \simeq \Q[T]/(c_\frakp(T)).
    \end{equation}
\end{lemma}

\begin{proof}
    Suppose that the CM for $J$ is defined over $F' \supseteq F$, so
    $\End(J_{F'})_\Q=L$.  Let $\frakp'$ be a prime above $\frakp$ in $F'$.
    Then by Oort \cite[(6.5.e)]{oort}, if $p$ splits in $L$ then $J_{\Fp}$ is
    ordinary, so $\End(J_{\F_{\frakp'}})_\Q=L$. Let $\pi \in \End(J)$ be the
    geometric Frobenius for $\frakp$ and similarly $\pi' \in \End(J_{F'})$ for
    $\frakp'$.  Then by Tate \cite[Theorem 2]{tate}, $\Q[\pi']=L$ and in
    particular the characteristic polynomial of $\pi'$ is irreducible.  But
    $\pi'$ is a power of $\pi$, so we have the inclusions $L \supseteq \Q[\pi]
    \supseteq \Q[\pi']=L$, and the lemma follows.
\end{proof}

We compute upper bounds on $\rho(J^{\alg})$ in the following way.  By
Proposition~\ref{prop:tate}(a), we can compute $\rho(J_{\frakp}^{\alg})$ for
many good primes $\frakp$ by counting points on $X_\frakp$.  We have two cases:
\begin{itemize}
    \item If $\rho(J^{\alg})$ is even, then by Charles \cite[Theorem
      1]{charles-picardk3} (part (2) cannot occur) there are infinitely many
      primes such that  $\rho(J^{\alg}) = \rho(J_{\frakp}^{\alg})$.
    \item If $\rho(J^{\alg})$ is odd, then also by Charles
      \cite[Proposition 18]{charles-picardk3} (in our setting we must have $E =
      \Q$), there are infinitely many pairs of primes $(\frakp_1, \frakp_2)$
      such that
      \begin{gather}
          \rho(J^{\alg}) + 1 = \rho(J_ {\F^{\alg} _{\frakp_1}})
          = \rho(J_ {\F^{\alg} _{\frakp_2}})\\
          \disc( \NS(J_ {\F^{\alg}_{\frakp_1}})) \not \equiv
          \disc( \NS(J_ {\F^{\alg}_{\frakp_2}})) \bmod{\Q^{\times 2}}.
          \label{eqn:dicsnscomp}
      \end{gather}
\end{itemize}
By \eqref{eqn:NsJfalg}, we then seek out the minimum values of
$\rho(J_{\F^{\alg}_{\frakp}})$ over the first few primes $\frakp$ of good
reduction; and for those where equality holds, we check \eqref{eqn:dicsnscomp}
using \eqref{disc}, improving our upper bound by $1$ when the congruence fails.
This upper bound for $\rho(J^{\alg})$ gives an upper bound for
$\End(J^{\alg})_\Q$ by \eqref{eqn:rhotoEnd}, and a guess for
$\End(J^{\alg})_\R$ except when $\rho(J^{\alg})=2$. For example, this approach
allows us to quickly rule out the possibility that $J^{\alg}$ has quaternionic
multiplication (QM) by showing that $\rho(J^{\alg}) \leq 2$.

To conclude, suppose that we are in the remaining case where, after many primes
$\frakp$, we compute $\rho(J^{\alg}) \leq 2$ and we believe that equality
holds.  Then the subalgebra $L_0 \subseteq \End(J^{\alg})_\Q$ fixed under the
Rosati involution has dimension $\leq 2$ over $\R$.  We proceed as follows.

\begin{enumerate}
    \item By the algorithms in the previous section, we can find and certify a
      nontrivial endomorphism.  So with a day-and-night algorithm, eventually
      either we will find $\rho(J^{\alg})=1$ or we will have certified that the
      Rosati-fixed endomorphism algebra $L_0$ is of dimension $2$.
    \item Next, we check if $L_0$ is a field by factoring the minimal
      polynomial of the endomorphism generating $L_0$ over $\Q$.  If $L_0
      \simeq \Q \times \Q$ splits, then by section \ref{sec:splitjac} we can
      split the Jacobian up to isogeny as the product of elliptic curves, and
      from there deduce the geometric endomorphism algebra and endomorphism
      ring.
    \item To conclude, suppose that $L_0$ is a (necessarily real) quadratic
      field.  Then by \eqref{eqn:rhotoEnd} we need to distinguish between RM
      and CM. We apply Lemma \ref{lem:endJFalg} to search for a candidate CM
      field or to rule out the CM possibility, by finding two nonisomorphic
      candidate CM fields. This approach is analogous to Lombardo's approach
      \cite[\S 6.3]{lombardo-endos},  and we refer to his work for a careful
      exposition.
\end{enumerate}

In practice, this method is very efficient to find sharp upper bounds, using
only a few small primes.

\begin{example}
    While computing the upper bound for all 66\,158 genus 2 curves in the LMFDB
    database, we only had to study their reductions for $p \leq 53$ and for
    more than 96\%  of the curves $p \leq 19$ was sufficient.  The unique curve
    requiring $p=53$ was the curve
    \href{http://www.lmfdb.org/Genus2Curve/Q/870400/a/870400/1}{\texttt{870400.a.870400.1}}:
    the prime $p=53$ is the first prime of good reduction for which the $2$
    elliptic curve factors are not geometrically isogenous modulo $p$.
    Altogether, computing these upper bounds took less than 7 CPU minutes.
\end{example}

\subsection{Endomorphism algebras over finite fields}

Starting in this section, we now consider upper bounds in higher genus. In this
section, we compute the dimension of the geometric endomorphism algebra of an
abelian variety over a finite field from the characteristic polynomial of
Frobenius.  We will apply this to reductions of an abelian variety in the next
sections.

First a bit of notation.  Let $R$ be a (commutative) domain and let $M$ be a
free $R$-module of finite rank $n$.  Let $\pi \in \End_R(M)$ be an $R$-linear
operator, and let
\begin{equation}
    c(T)=\det(1-\pi T\,|\, M) \in 1+TR[T]
\end{equation}
be its characteristic polynomial acting on $M$.  For $r \geq 1$, we define
\begin{equation} \label{eqn:crTtwice}
    \begin{aligned}
        c^{(r)}(T) &\colonequals \det(1-\pi^r T\,|\, M) \\
        c^{\otimes r}(T) &\colonequals \det(1-\pi^{\otimes r}T \,|\, M^{\otimes r})
    \end{aligned}
\end{equation}
We have $\deg c^{(r)}(T)=\deg c(T)=n$ and $\deg c^{\otimes r}(T)=n^r$.  If
$c(T)=\prod_{i=1}^{n}(1-z_i T)$ with $z_i \in R$, then
$c^{(r)}(T)=\prod_{i=1}^n (1-z_i^r T)$ and
\begin{equation}
    c^{\otimes r}(T)=\prod_{1 \leq i_1,\dots,i_r \leq n} (1-z_{i_1}\cdots z_{i_r} T).
\end{equation}
We may compute $c^{\otimes 2}$ as a polynomial resultant
\begin{equation}
    c^{\otimes 2} (T) \colonequals \operatorname{Res}_{z} ( c(z) , z^n c(T/z) )
    \in \Z[T].
\end{equation}

Let $A$ be an abelian variety over the finite field $\F_q$ with $\dim A=g$, let
$A^{\alg} = A_{\F_q^{\alg}}$ be its base change to an algebraic closure
$\F_q^{\alg}$, and let $\Frob_q$ be the $q$-power Frobenius automorphism. We
write
\begin{equation}
    c(T) \colonequals
    \det(1-\Frob_q T \,|\, H_{\et}^1(A^{\alg},\Q_\ell)) \in 1+T\Z[T]
\end{equation}
for a prime $\ell \nmid q$ (with $c (T)$ independent of $\ell$). Then $\deg
c=2g$.  By the Riemann hypothesis (a theorem in this setting), the reciprocal
roots of the polynomial $c^{\otimes 2}(T)$ have complex absolute value $q$.
Factor
\begin{equation}
    \label{eqn:fpfact}
    c^{\otimes 2} (  T ) = h(T) \prod_{i} \Phi_{k_i}( q T),
\end{equation}
over $\Z[T]$ where $\Phi_{k_i}(T)$ is a cyclotomic polynomials (the minimal
polynomial of a primitive $k_i$th root of unity) for each $i$ and $h(T)$ is a
polynomial with no reciprocal roots of the form $q$ times a root of unity.  (In
the factorization \eqref{eqn:fpfact}, we allow repetition $k_i=k_j$ for $i \neq
j$.)

We now recall a consequence of the (proven) Tate conjecture suitable for our
algorithmic purposes.

\begin{lemma}\label{lem:endfq}
    The following statements hold.

    \begin{enumalph}
    \item For all $r \geq 1$, factoring as in \eqref{eqn:fpfact} we have
        \begin{equation}
            \dim_\Q \End\bigl(A_{\F_{q^r}} \bigr)_\Q
            = \sum_{ k_i \mid r } \deg \Phi_{k_i}.
        \end{equation}
    \item Let $k \colonequals \lcm \{k_i\}_i$. Then $\F_{q^k}$ is the minimal
      field over which $\End(A^{\alg})$ is defined.
    \end{enumalph}
\end{lemma}
\begin{proof}
    Let $r \geq 1$.
    The Tate conjecture (proven by Tate \cite[Theorem 4]{tate}) applied to
    $A \times A$ over $\F_{q^r}$  implies
    \begin{equation}
        \End \bigl(A_{\F_{q^r}} \bigr) \otimes \Q_\ell \simeq
        \bigl(
            H_{\et}^1\bigl(A^{\alg}, \Q_\ell \bigr) ^{\otimes 2}  (1)
        \bigr)^{\Gal(\F^{\alg}_q\,|\, \F_{q^r})}.
    \end{equation}
    Factoring $c(T)=\prod_i (1-z_i T) \in \C[T]$, so that $c^{\otimes 2}(T) =
    \prod_{i,j} (1-z_i z_j T)$, we have
    \begin{equation} \label{eqn:dimQcyc}
        \begin{aligned}
            \dim_\Q \End\bigl(A_{\F_{q^r}} \bigr)_\Q
            &= \# \bigl\{ (i,j) : (z_i z_j)^r = q^r \bigr\}\\
            &= \# \bigl\{ (i,j) : z_i z_j
        = \zeta q \text{ with } \zeta^r=1\bigr\} \\
        &= \sum_{ k_i \mid r } \deg \Phi_{k_i}.
        \end{aligned}
    \end{equation}
    (Working over $\F_q$, in Tate's notation we have $\dim_\Q
    \End(A_{\F_{q}})_\Q = r(f,f)$, where $f$ is the characteristic polynomial of Frobenius and $r(f,f)$ is
    the multiplicity of the root $q$ in $f^{\otimes 2}$.)  This proves (a).

    The sum in \eqref{eqn:dimQcyc} attains its maximum value for the first time
    when $r=k = \lcm_i \{k_i\}_i$, and by maximality we have $\End \bigl(
    A^{\alg} \bigr) =  \End\bigl(A_{\F_{q^k}}\bigr)$, which proves (b).
\end{proof}

We will make use of the following more specialized statement.

\begin{corollary} \label{cor:dimQdimA}
    Suppose that $c(T)$ is separable and that the subgroup of $\Qbar^\times$
    generated by the (reciprocal) roots is torsion free.  Then all
    endomorphisms of $A^{\alg}$ are defined over $\F_q$ (i.e., $k=1$) and
    \begin{equation*}
        \dim_\Q \End(A)_\Q = 2\dim A.
    \end{equation*}
\end{corollary}

\begin{proof}
    Factor $c(T)=\prod_{i=1}^{2\dim A} (1-z_i T)$ over $\Qbar$.  Since $c(T)$
    is separable, its reciprocal roots $z_i \in \Qbar^\times$ are distinct.
    Suppose that $z_iz_j=\zeta q$ where $\zeta$ is a root of unity.  By the
    (proven) Riemann hypothesis for abelian varieties, associated to $z_i$ is a
    reciprocal root $z_{i'}$ such that $z_iz_{i'}=q$.  By separability, the
    index $i'$ is uniquely determined by $i$.

    We now have $z_j/z_{i'}=\zeta$.  Therefore $\zeta=1$ since the subgroup
    generated by the roots is torsion free.  By distinctness of the roots we
    obtain $i'=j$.  So among the reciprocal roots $z_i z_j$ of $c^{\otimes
    2}(T)$ there are exactly $2\dim A$ pairs $(i,j)$ with $z_i z_j=q$.  We have
    shown that in the factorization \eqref{eqn:fpfact} there are $2\dim A$
    factors $\Phi_1(qT) = 1-qT$, all with $k_i=1$, and no other cyclotomic
    factors.  The result then holds by Lemma \ref{lem:endfq}(a).
\end{proof}

Lemma \ref{lem:endfq} immediately implies that
\begin{equation}
    \dim_\Q \End\bigl(A^{\alg})_\Q = \sum_i \deg \Phi_{k_i}
\end{equation}
so we have direct access to the dimension of the geometric endomorphism algebra
from the characteristic polynomial of Frobenius.

\begin{remark}
  \label{remark:hondatate}
    Although we will not use this in what follows, we can upgrade
    Lemma~\ref{lem:endfq} from dimensions to a full description of the
    endomorphism algebra itself up to isomorphism by the use of Honda--Tate
    theory, as follows. Factor
    \begin{equation}
        c(T) = c_1(T)^{m_1} \cdots c_t(T)^{m_t}
    \end{equation}
    where $c_i$ are distinct, irreducible polynomials in $\Z[T]$. Applying
    Honda--Tate theory, see for example Waterhouse \cite[Chapter
    2]{waterhouse-69} and Waterhouse--Milne \cite[Theorem
    8]{waterhouse-milne-71}, each irreducible polynomial $c_i(T)$ determines
    (by the $p$-adic valuation of its coefficients) a division algebra $B_i$
    over $\Q$ such that $B_i$ is central over the field $L_i \colonequals
    \Q[T]/(c_i(T))$ and $e_i^2 \colonequals \dim_{L_i} B_i$ has $e_i \mid m_i$;
    these combine to give
    \begin{equation}
        \End(A)_\Q \simeq B_1^{n_1} \times \dots \times B_t^{n_t}
    \end{equation}
    where $n_i=m_i/e_i$. This decomposition of the endomorphism algebra
    corresponds to the decomposition of $A$ up to isogeny over $\F_q$ as
    \begin{equation}
        A \sim A_1^{n_1} \times \dots \times A_t^{n_t}
    \end{equation}
    where the abelian varieties $A_i$ over $\F_q$ are simple and pairwise
    nonisogenous over $\F_q$, and $\End(A_i)_\Q \simeq B_i$.

    We can apply this to compute the structure of the geometric endomorphism
    algebra by computing $k$ as in Lemma \ref{lem:endfq} and applying the above
    to $c^{(k)}(T)$.
\end{remark}

To conclude, we extract another description of the dimension of the
endomorphism algebra, again due to Tate.

\begin{lemma} \label{lem:tateic}
    Factor $c(T)=\prod_{i}^s (1-z_i T)^{m(z_i)} \in \C[T]$ with $z_i \in \C$
    distinct.  Then
    \begin{equation}
        \dim_\Q \End(A)_\Q = \sum_{i=1}^s m(z_i)^2.
    \end{equation}
\end{lemma}

\begin{proof}
    See Tate \cite[Theorem 1(a), Proof of Theorem 2(b)]{tate}; in his notation
    $f = c$ and the right hand side is $r(f,f)$.
\end{proof}

\subsection{Upper bounds in higher genus: decomposition into powers}

In the next two sections, we discuss how to produce tight upper bounds on the
dimension of the geometric endomorphism algebra for a general abelian variety,
under certain hypotheses. Our approach will be analogous to
section~\ref{sec:71g2}, however, instead of studying the reduction homomorphism
induced on the N\'{e}ron--Severi lattice, we will study the reduction
homomorphism induced on the endomorphism rings themselves.

These bounds come in two phases.  In the first phase, described in this
section, we describe a decomposition of an abelian variety over a number field
into powers of geometrically simple abelian varieties. In the next phase,
described in the next section, we refine this decomposition to bound the
dimension of the geometric endomorphism algebra by examination of the center.

We work in slightly more generality in these two sections than in the rest of
the paper. Let $A$ be an abelian variety of a number field $F$ (not necessarily
the Jacobian of a curve). Let $F_A$ be the minimal field over which $\End(A)$
is defined. Let $\frakp$ be a nonzero prime of (the ring of integers) of $F$,
let $\F_\frakp$ be its residue field with $q=\#\F_\frakp$, and let
$\Frob_\frakp$ be the $q$-power Frobenius automorphism. For $r \geq 1$, we
denote $\F_{\frakp^r} \supseteq \F_\frakp$ the finite extension of degree $r$
in an algebraic closure $\F_{\frakp}^{\alg}$.

Suppose that $A$ has good reduction $A_{\frakp}$ at $\frakp$.  Write
\begin{equation}  \label{eqn:cpTdef}
    c_{\frakp}(T) \colonequals
    \det(1 - \Frob_\frakp T\,|\, H_{\et}^1(A_{\frakp}^{\alg},\Q_\ell))
    \in 1+T\Z[T]
\end{equation}
for the characteristic polynomial of $\Frob_\frakp$ acting on the first
$\ell$-adic \'etale cohomology group (independent of $\ell \nmid q$).

The reduction (specialization) of an endomorphism modulo $\frakp$ induces an
injective ring homomorphism
\begin{equation}
    s_\frakp\colon \End(A^{\alg}) \hookrightarrow \End(A_{\frakp}^{\alg}).
\end{equation}
Therefore $\dim_\Q \End(A^{\alg})_\Q \leq \dim_\Q \End(A_{\frakp}^{\alg})_\Q$.
However, unless $A$ is a CM abelian variety, this inequality will always be
strict, so we undertake a more careful analysis.

Up to isogeny over $F_A$, we factor
\begin{equation} \label{eqn:AFAdecomp}
    A_{F_A} \sim A_1 ^{n_1} \times \cdots \times A_t ^{n_t}
\end{equation}
where $A_i$ are geometrically simple and pairwise nonisogenous abelian
varieties over $F_A$. Let $B_i \colonequals \End(A_i^{\alg})_\Q$ be the
geometric endomorphism algebra of $A_i$, let $L_i \colonequals Z(B_i)$ be the
center of $B_i$, and write $e_i ^2  \colonequals \dim_{L_i} (B_i)$ with $e_i
\in \Z_{\geq 1}$. We have
\begin{equation} \label{eqn:AalgB}
    \End(A^{\alg})_\Q \simeq \M_{n_1}(B_1) \times \dots \times \M_{n_t}(B_t).
\end{equation}

For the prime $\frakp$, we define $k_\frakp$ to be the smallest integer such
that $\End(A_\frakp^{\alg})$ is defined over $\F_{\frakp^{k_\frakp}}$. The
polynomial $c_{\frakp,i}^{(k_\frakp)}(T)$ (see  \eqref{eqn:crTtwice}) is the
Frobenius polynomial for $A_i$ over $\F_{\frakp^{k_\frakp}}$.

\begin{proposition}    \label{prop:lowerbound}
    The following statements hold.
    \begin{enumalph}
    \item For every $i=1,\dots,t$, there exists $g_{\frakp,i}(T) \in 1 + T \Z[T]$
        such that
        \begin{equation} \label{prop:lowerboundeqa}
            c_{\frakp,i}^{(k_\frakp)}(T) = g_{\frakp,i}(T)^{e_i}.
        \end{equation}
    \item We have
        \begin{equation} \label{eqn:fundineqAis}
            2 \sum_{i=1}^t e_i n_i^2 \dim A_i =
            \sum_{i = 1} ^t e_i ^2 n_i ^2 \deg g_{\frakp,i} \leq \dim_\Q \End(A_\frakp^{\alg})_\Q
        \end{equation}
        and equality is obtained in \eqref{eqn:fundineqAis} if and only if the
        polynomials $g_{\frakp,i}(T)$ in \textup{(a)} are separable and
        pairwise coprime.
    \end{enumalph}
\end{proposition}

\begin{proof}
    We begin by proving (a), and for this purpose we may assume $A=A_i$.
    Following Zywina \cite[\S 2.3]{zywina-14}, let $\FAconn$ be the smallest
    extension of $F$ such that the $\ell$-adic monodromy group associated to
    $A$ is connected over $\FAconn$. Then $\FAconn$ is Galois over $F$ and
    $\FAconn \supseteq F_A$, so all endomorphisms of $A$ are defined over
    $\FAconn$.  If $F=\FAconn$, then (a) is proven by Zywina \cite[Lemma
    6.3(i)]{zywina-14}: part (i) (but not the rest of his Lemma~6.3) only needs
    the hypothesis that $\frakp$ is a prime of good reduction. The general case
    follows by applying the previous sentence to a prime in $\FAconn$ lying
    above $\frakp$.

    Now we prove (b).  We first treat the case where $A=A_i$ is geometrically
    simple, and we drop the subscript $i$. Factor
    \begin{equation}
        g_{\frakp}(T)=\prod_{j} (1 - \gamma_j T)^{m(\gamma_j)} \in \C[T]
    \end{equation}
    where the reciprocal roots $\gamma_j$ are pairwise distinct and occur with
    multiplicity $m(\gamma_j)$.  Then
    \begin{equation}
        \deg g_{\frakp} = \sum_{j} m(\gamma_{j})  \leq \sum_{j} m(\gamma_{j})^2
    \end{equation}
    and the equality is attained if and only if $m(\gamma_j) = 1$ for all $j$,
    in other words, if $g_{\frakp,j}(T)$ is separable. By
    Lemma~\ref{lem:tateic}, since $c^{(k_\frakp)} _\frakp (T)  = g_\frakp (T)
    ^e$, we have
    \begin{equation}
      \dim_\Q \End(A_\frakp^{\alg})_\Q  =  \sum_{j} (e m(\gamma_{j}) )^2.
    \end{equation}
    Since $2\dim A = \deg c_\frakp^{(k_\frakp)} = e \deg g_\frakp$ we conclude
    \begin{equation}
      \label{eqn:inequalitygeosimple}
        2e\dim A = e^2 \deg g_{\frakp} \leq e^2 \sum_j m(\gamma_j)^2 = \dim_\Q \End(A_\frakp^{\alg})_\Q
    \end{equation}
    as claimed.

    Now we treat the general case. For $i=1,\dots,t$, in the notation of part
    (a), factor
    \begin{equation}
        g_{\frakp,i}(T)=\prod_{j}^s (1 - \gamma_{ij} T)^{m(\gamma_{ij})} \in \C[T].
    \end{equation}
    Adding up the inequality \eqref{eqn:inequalitygeosimple}, multiplied by
    $n_i ^2$ throughout, we obtain
    \begin{equation}
        \sum_i e_i ^2 n_i ^2 \deg g_{\frakp, i} \leq \sum_i e_i ^2 n_i ^2 \sum_j m(\gamma_{ij})^2 \leq \dim_\Q \End(A_\frakp^{\alg})_\Q.
    \end{equation}
    As in the previous paragraph, the left-hand inequality is an equality if
    and only if for every $i$ the polynomial $ g_{\frakp, i}(T)$ is separable.
    By Lemma \ref{lem:tateic}, the right-hand inequality is an equality if
    and only if
    \begin{equation}
      z_\ell = \gamma_{ij} \quad \text{and} \quad m(z_\ell) =  e_i n_i m(\gamma_{ij}), 
   \end{equation}
   where
   \begin{equation}
     c_\frakp ^{(k)} (T) = \prod_\ell (1 - z_\ell T) ^{m(z_\ell)}
   \end{equation}
   with $z_\ell$ distinct. In other words, if and only if $\gamma_{ij}$ are all
   distinct, or equivalently, if the polynomials $g_{\frakp,i}$ are pairwise
   coprime.
\end{proof}

We now try to deduce the decomposition \eqref{eqn:AFAdecomp} from
factorizations as in Proposition~\ref{prop:lowerbound}.  From the left-hand
side of \eqref{eqn:fundineqAis}, we define the quantity
\begin{equation}
    \eta(A^{\alg}) \colonequals  \sum_{i=1}^t e_i n_i^2 \dim A_i
\end{equation}
which we would like to know. (The invariant $\eta(A^{\alg})$ plays a similar
role to that of $\rho(J^{\alg})$ in the previous section.) Looking at the
right-hand side of \eqref{eqn:fundineqAis}, for a prime $\frakp$ of good
reduction, we factor
\begin{equation}
    c_{\frakp}^{(k_\frakp)}(T) = \prod_{i=1}^{t_\frakp} h_{\frakp,i}(T)^{m_{\frakp,i}} \in \Z[T]
\end{equation}
into pairwise coprime irreducibles, where $k_\frakp$ can be computed with
Lemma~\ref{lem:endfq} and $t_\frakp$ is the number of pairwise distinct
(simple) factors in the isogeny decomposition of $A_{\F_\frakp^{k_\frakp}}$.
Now we define the computable quantity
\begin{equation}  \label{eqn:computeetap}
    \eta(A^{\alg}_\frakp) \colonequals \sum_{i=1}^{t_\frakp} m_{\frakp,i}^2 \deg h_{\frakp,i} \in 2\Z_{\geq 1}.
\end{equation}
It follows from Lemma \ref{lem:tateic} that $\eta(A^{\alg}_\frakp)= \dim_\Q
\End(A_\frakp^{\alg})_\Q$.

\begin{corollary} \label{cor:inequaleta}
\begin{enumalph}
\item
    For all good primes $\frakp$, we have
    \begin{equation} \label{eqn:2etaaalg}
        \eta(A^{\alg}) \leq \tfrac{1}{2}\eta(A^{\alg} _\frakp).
    \end{equation}
\item If equality holds in \eqref{eqn:2etaaalg}, then $t \leq t_\frakp$.
\item If equality holds in \eqref{eqn:2etaaalg} and $t=t_\frakp$, then
as multisets
    \begin{equation}
        \{(m_{\frakp,i},\tfrac{1}{2}m_{\frakp,i}\deg
h_{\frakp,i})\}_{i=1}^{t_\frakp} = \{(e_in_i, n_i\dim A_i)\}_{i=1}^t.
    \end{equation}
\end{enumalph}
\end{corollary}

\begin{proof}
    The inequality $\eta(A^{\alg}) \leq \tfrac{1}{2}\eta(A^{\alg} _\frakp)$ is
    simply rewriting \eqref{eqn:fundineqAis}.

    If equality holds in \eqref{eqn:2etaaalg}, then by
    Proposition~\ref{prop:lowerbound}(b), the $t$ polynomials $g_{\frakp,i}(T)$
    are separable and pairwise coprime, and there are $t_\frakp$ distinct
    factors in total.  Hence, $t \leq t_\frakp$.

    Moreover, if equality holds in \eqref{eqn:2etaaalg} and $t = t_\frakp$, then we have two factorizations of $c_{\frakp}^{(k_\frakp)}(T)$ into
    powers of pairwise irreducibles, one in terms of $g_{\frakp,i}(T)$ and
    the other in terms of $h_{\frakp,i}(T)$.
    Therefore, as a multiset we have
    \begin{equation}
        \{ (m_{\frakp,i}, h_{\frakp,i}(T)) \}_{i=1} ^{t_\frakp} =  \{ e_i n_i , g_{\frakp,i}(T)) \}_{i=1} ^{t}.
    \end{equation}
    Taking the degree of the second entry and multiplying it by the first
    entry, we get
    \begin{equation}
        \{ (m_{\frakp,i}, m_{\frakp,i} \deg h_{\frakp,i}) \}_{i=1} ^{t_\frakp}
        =  \{ (e_i n_i , e_i n_i \deg g_{\frakp,i}) \}_{i=1} ^{t},
    \end{equation}
    and the desired equality follows by noting that
    $2 \dim A_i = e_i \deg g_{\frakp,i}$.
\end{proof}

We now show that (conjecturally) there are an abundance of primes where we have
an equality in \eqref{eqn:2etaaalg}, i.e., primes for which the endomorphism
algebra grows in a controlled (minimal) way under reduction modulo $\frakp$.

Let $S$ be the set of primes $\frakp$ of $F$ with the following properties:
\begin{enumroman}
  \item The prime $\frakp$ is a prime of good reduction for $A$.
  \item $\Nm(\frakp)$ is prime, i.e., the residue field $\#\F_\frakp$ has prime
    cardinality.
  \item $\End(A_\frakp^{\alg})$ is defined over $\F_\frakp$ (i.e.,
    $k_\frakp=1$).
  \item For all $i=1,\dots,t$, we have an isogeny $(A_i)_{\frakp} \sim
    A_{\frakp,i}^{e_i}$ over $\F_\frakp$ where $A_{\frakp,i}$ is simple;
    moreover, the abelian varieties $A_{\frakp,i}$ are pairwise nonisogenous.
  \item For all $i=1,\dots,t$, the algebra $\End(A_{\frakp,i} ^{\alg})_\Q$ is a
    field, generated by the Frobenius endomorphism.
\end{enumroman}

\begin{lemma} \label{lem:pSspicy}
    For $\frakp \in S$, we have $t=t_\frakp$ and $2\eta(A^{\alg}) =
    \eta(A^{\alg} _\frakp)$.
\end{lemma}

\begin{proof}
    We have $t=t_\frakp$ by the decomposition in (iv). By
    \eqref{prop:lowerboundeqa}, $g_{\frakp,i}(T)$ is the characteristic
    polynomial of Frobenius for $A_{\frakp,i}$, and for every $i$, the
    polynomials $g_{\frakp,i}(T)$ are irreducible over $\Q$, otherwise the
    Honda--Tate theory would give a further splitting of $A_\frakp$, see
    Remark~\ref{remark:hondatate}. and pairwise coprime by (iv), so the
    equality holds by Proposition~\ref{prop:lowerbound}.
\end{proof}

The required analytic result about primes $\frakp \in S$ is essentially proved
by Zywina \cite{zywina-14}, as follows. For the statement of the Mumford--Tate
conjecture, see Zywina \cite[\S 2.5]{zywina-14} and the references given;
although the conjecture is still open, many general classes of abelian
varieties are known to satisfy the conjecture.

\begin{proposition}
    \label{prop:zywina-extended}
    Suppose that the Mumford--Tate conjecture for $A$ holds. Then the set $S$
    has positive density.
\end{proposition}

\begin{proof}
    We follow the proof of a result by Zywina \cite[Theorem 1.4]{zywina-14}.
    He shows that the set of primes with properties (i)--(iv) has positive
    density, and we obtain our full result by a refining his proof to obtain
    property (v) as a consequence, as follows.

    As in the proof of Proposition \ref{prop:lowerbound}, we first suppose that
    $F=\FAconn$. Zywina \cite[Section 2.4]{zywina-14} considers the set of
    primes satisfying (i)--(ii) and such that the Frobenius eigenvalues of each
    $A_{\frakp,i}$ generate a torsion-free subgroup of maximal rank in
    $(\Q^{\alg})^{\times}$. Zywina shows that this set has density $1$ and
    proves \cite[Lemma 6.3]{zywina-14} that the Frobenius eigenvalues of
    $A_{\frakp,i}$ are distinct.  Next, among primes in this set, away from
    a set of primes of density zero \cite[Proposition 6.6]{zywina-14},
    the characteristic polynomial of Frobenius on $A_{\frakp,i}$ is irreducible
    \cite[Lemma 6.7]{zywina-14}, so that (iv) holds. The hypotheses of
    Corollary \ref{cor:dimQdimA} hold for the abelian variety $A_{\frakp,i}$
    over $\F_\frakp$, so all endomorphisms are defined over $\F_\frakp$, so
    (iii) holds, and moreover $\dim_\Q \End(A_{\frakp,i})_\Q = 2\dim
    A_{\frakp,i}$.  Finally, since $A_{\frakp,i}$ is simple, the characteristic
    polynomial $g_{\frakp,i}(T)$ of Frobenius is irreducible with $\deg
    g_{\frakp,i}(T)=2 \dim A_{\frakp,i}$, and so
    $\End(A_{\frakp,i}^{\alg})_\Q=\End(A_{\frakp,i} )_{\Q}$ contains as a
    subalgebra the field $\Q[T]/(g_{\frakp, i}(T))$ generated by Frobenius.  By
    dimension counts, equality holds and (v) follows.

    The general case follows by applying this argument to the set of primes
    $\frakp$ of $F$ such that a prime above $\frakp$ in the Galois extension
    $\FAconn$ over $F$ are in the above set: by the Chebotarev density theorem,
    this set has density $[\FAconn : F]^{-1}$.
\end{proof}

\begin{proposition} \label{prop:followingeffcomp}
    If the Mumford--Tate conjecture holds for $A$, then the following
    quantities are effectively computable:
    \begin{enumroman}
        \item The integer $\eta(A^{\alg})$;
        \item The number $t$ of geometrically simple factors of $A$; and
        \item The set of tuples $\{(e_in_i, n_i \dim A_i)\}_{i=1}^t$.
    \end{enumroman}
\end{proposition}

\begin{proof}
    We pursue a day-and-night approach. By day, we search for endomorphisms of
    $A^{\alg}$, find a (partial) decomposition
    \begin{equation}
        A_L \sim (A_1')^{n_1'} \times \dots \times (A_{t'})^{n_{t'}'},
    \end{equation}
    and compute the quantity
    \begin{equation} \label{eqn:eit}
        \eta' = \sum_{i=1}^{t'} e_i' (n_i')^2 \dim A_i' \leq \eta(A^{\alg}).
    \end{equation}
    By night, by counting points on $A_\frakp$ we compute $t_\frakp$ and
    $\eta(A_\frakp^{\alg})$ for many good primes $\frakp$ using
    \eqref{eqn:computeetap}. We continue in this way until we find a prime
    $\frakp$ such that $t'=t_\frakp$ and $2\eta' = \eta(A_\frakp^{\alg})$:
    Proposition~\ref{prop:zywina-extended} and Lemma \ref{lem:pSspicy} assure
    us that this will happen frequently, proving that the quantities (i) and
    (ii) are effectively computable.  For (iii), we then appeal to
    Corollary~\ref{cor:inequaleta}.
\end{proof}

\begin{remark} \label{rmk:propothermt}
    The statement of Proposition~\ref{prop:followingeffcomp} can be proven in
    other ways without the Mumford--Tate conjecture, but the algorithm
    exhibited in the proof is quite practical! For example, we can expect to
    verify that the abelian variety $A$ over $F$ is geometrically a power of a
    simple abelian variety (so $t=t_\frakp=1$) after examining $[\FAconn:F]$
    Frobenius polynomials.
\end{remark}

\subsection{Upper bounds in higher genus: bounding the center}

Now we refine the decomposition obtained in the previous section to bound the
dimension of the geometric endomorphism algebra, by bounding the center. For a
theoretical result, we refer again to Lombardo \cite[\S 5]{lombardo-endos}. Our
method in this section are heuristic in nature, as we lack an analytic result
(Hypothesis~\ref{hyp:center}) that assures us that our method terminates. That
being said, our method is efficient, and if our method terminates then the
output will be correct.

We start with some preliminary lemmas.

\begin{lemma} \label{lem:LBFN}
    Let $B$ be a central simple algebra over $F$ and let $N$ be a finite
    extension of $F$.  Then there exists a maximal subfield $L \subseteq B$
    such that $L$ and $N$ are linearly disjoint over $F$.
\end{lemma}

\begin{proof}
    We may assume $N$ is Galois over $F$ and exhibit $L$ such that $L \cap N =
    F$. As a consequence of the Albert--Brauer--Hasse--Noether theorem (see
    e.g.\ Reiner \cite[Theorem 32.15]{MR1972204}), we have an embedding $L
    \hookrightarrow B$ if and only if $L$ satisfies finitely many local
    conditions (determined by the ramified primes in $B$). Choose a prime
    $\frakp$ of $F$ that splits completely in $N$ disjoint from these finitely
    many conditions, and add the new local condition that $L_\frakp$ is a
    field.  By Chebotarev there are infinitely many fields $L$ satisfying
    these conditions, and in $L \cap N$ we have $\frakp$ both splitting
    completely and with a unique prime above it, so $F=L \cap N$.
\end{proof}

\begin{lemma}\label{lemma:center}
    For $i=1,2$, let $K_i$ be a number field and let $B_i$ be a central simple
    algebra over $K_i$.  Let $\varphi\colon B_1 \hookrightarrow B_2$ be an
    injective $\Q$-algebra homomorphism. If $\dim_{K_1} B_1 = \dim_{K_2} B_2 =
    m^2$, then $\varphi(K_1) \subseteq K_2$.
\end{lemma}

\begin{proof}
    Let $L$ be a maximal subfield of $B_1$, so $\dim_{K_1} L = m$. Thus
    $\varphi(L)$ is a subfield of $B_2$, and $S \colonequals K_2 \varphi(L)
    \subset B_2$, the subring of $B_2$ generated by $K_2$ and $\varphi(L)$, is
    a commutative $K_2$-subalgebra of $B_2$. Hence, $\dim_{K_2} S \leq m$.

    Let $E = \varphi(K_1) \cap K_2 \subset B_2$ and $N_2$ be the normal closure
    of $\varphi(K_1) K_2$. We may choose $L$ such that $\varphi(L)$ and $N_2$
    are linearly disjoint by Lemma \ref{lem:LBFN}. Thus $\varphi(L) \otimes_E
    K_2 \xrightarrow{\sim} S$ and
    \begin{equation}
        \begin{aligned}
            m \geq \dim_{K_2} S &= \frac{\dim_E S}{\dim_E K_2}
            = \frac{\dim_E \varphi(L) \dim_E K_2}{\dim_E K_2}\\
            &= \dim_E \varphi(L)
            = \dim_{\varphi(K_1)} \varphi(L) \dim_E \varphi(K_1)  \geq m.
        \end{aligned}
    \end{equation}
    Therefore, $\dim_E \varphi(K_1) = 1$ and $\varphi(K_1) \subseteq K_2$.
\end{proof}

We now recall the notation described in the previous section (starting with
\eqref{eqn:AFAdecomp}), in particular $L_i =Z(B_i)$.

\begin{corollary}\label{corollary:center}
    If the polynomial $g_{i, \frakp}(T)$ in
    Proposition~\ref{prop:lowerbound}\textup{(a)} is irreducible for some $i$,
    then $L_i$ is isomorphic to a subfield of $\Q[T]/ (g_{\frakp,i } (T))$.
\end{corollary}

\begin{proof}
    Apply Lemma~\ref{lemma:center} to the specialization homomorphism
    \begin{equation}
        s_{i,\frakp}\colon \End(A_i^{\alg})_\Q \hookrightarrow \End((A_i)_\frakp^{\alg})_\Q;
    \end{equation}
    On the left we have center $L_i$ and on the right we have center
    $\Q[T]/(g_{\frakp,i}(T))$ by Tate \cite[Theorem 2(a)]{tate}.
\end{proof}

Now we address the hypothesis that will allow us to deduce the candidate
fields for the centers.

\begin{hypothesis} \label{hyp:center}
    For every $i=1,\dots,t$, there exists a nonempty, finite collection of
    primes $\frakp_{ij} \in S$ such that $K$ is a subfield of $\Q[T]/
    (g_{\frakp_{ij}} (T))$ for all $j$ if and only if $K$ is a subfield of
    $L_i$.
\end{hypothesis}

The hypothesis is known to hold for abelian surfaces by an explicit argument of
Lombardo \cite[Theorem~6.10]{lombardo-endos}. In our experiments with higher
genus curves, every Jacobian variety we saw satisfied Hypothesis
\ref{hyp:center}.

\begin{proposition} \label{prop:centereffcomp}
    If the Mumford--Tate Conjecture and Hypothesis~\textup{\ref{hyp:center}} hold
    for $A$, then the centers $L_i$ are effectively computable.
\end{proposition}

\begin{proof}
    We continue with a day-and-night approach as described in the proof of
    Proposition~\ref{prop:followingeffcomp}.  The decomposition of $A$ by day
    into $t$ factors allows us to decompose $A[\ell ^r] \simeq
    (\Z/\ell\Z)^{2\dim A}$ into $t$ factors, and for $r$ large enough we can
    keep track of the index $i$ between different primes $\frakp$: see Lombardo
    \cite[Lemma 5.2]{lombardo-endos}. By night, we will have encountered an
    abundance of primes $\frakp$ such that $t=t_\frakp$ and
    $2\eta(A^{\alg})=\eta(A_\frakp^{\alg})$, and $c_\frakp(T)$ factors as
    \begin{equation}
        g_{\frakp, i}(T) ^{e_1 n_1} \cdots g_{\frakp, t} (T) ^{e_t n_t}
    \end{equation}
    with the polynomials $g_{\frakp,i}$ irreducible and pairwise coprime.

    For these primes $\frakp$, by Corollary \ref{corollary:center} we have $L_i
    \hookrightarrow \Q[T]/(g_{\frakp,i}(T))$, immediately giving only finitely
    many possibilities for each $L_i$. Finally, by Hypothesis~\ref{hyp:center}
    the only field that embeds in all $\Q[T]/(g_{\frakp,i}(T))$ is $L_i$, and
    so by computing intersections of subfields we will eventually find $L_i$.
\end{proof}

\begin{remark}
    Parallel to Remark~\ref{rmk:propothermt}, in practice the algorithm of
    Proposition~\ref{prop:centereffcomp} performs very well. In most cases, the
    abelian variety is geometrically a power and $L_i=\Q$, and in practice this
    can be quickly deduced by simply computing that the greatest common divisor
    of the discriminants $\disc \Q[T]/(g_{\frakp,i}(T))$ is equal to $1$.
\end{remark}

\begin{remark}
    Many cohomological algorithms for counting points on a curve can be adapted
    to keep track of the index $i$ between different primes $\frakp$ rather
    than resorting to the $\ell$-adic representation. For example, in those
    point counting algorithms that employ Monsky--Washnitzer cohomology, we may
    choose a basis of differentials that works for all good primes $\frakp$ and
    use the decomposition of $A$ into factors to decompose these differentials
    and thereby compute the action of Frobenius on each component. A similar
    argument applies to methods that compute the Hasse--Witt matrices.
\end{remark}

\section{Examples} \label{sec:examples}

We now give some further explicit illustrations of the methods developed above.

\subsection{Examples in genus 2}

\begin{example}
    We begin with the curve of genus 2 with LMFDB label
    \href{http://www.lmfdb.org/Genus2Curve/Q/12500/a/12500/1}{\texttt{12500.a.12500.1}},
    the smallest curve with potential RM in the LMFDB.  For convenience, we
    complete the square from the minimal Weierstrass model and work with the
    equation
    \begin{equation}
        X \colon y^2 = 5 x^6 + 10 x^3 - 4 x + 1 = f(x)
    \end{equation}
    so that $X \times X$ has affine patch described by $y_i^2=f(x_i)$ with
    $i=1,2$.

    Let $\alpha$ be a root of the polynomial $x^2 - x - 1$. Then we certify that
    the endomorphism ring of $X$ is the maximal order in the quadratic field $\Q
    (\alpha)$ of discriminant $5$.
    With basis of differentials $\d x/y, x\d x/y$, a generator has tangent
    representation
    $\begin{pmatrix} -\alpha & 0 \\ 0 & \alpha - 1 \end{pmatrix}$.
    For the base point $P_0 = (0, 1)$ a corresponding
    divisor in $X \times X$ is defined by the ideal
    \begin{equation}
        \begin{aligned}
            &\langle (2 \alpha - 1) x_1^2 x_2^2 - (\alpha + 2) x_1^2 x_2 + x_1^2
            - (\alpha + 2) x_1 x_2^2 + \alpha x_1 x_2 + x_2^2, \\
            &\quad (3 \alpha + 1) x_1^2 x_2 y_2
            - (2 \alpha + 4) x_1^2 y_2
            - (3 \alpha + 1) x_1 y_1 x_2^2
            + (4 \alpha + 3) x_1 y_1 x_2
            \\
            &\qquad - (\alpha - 1) x_1 y_1
            - (4 \alpha + 3) x_1 x_2 y_2 + (\alpha - 1) x_1 y_2
            + (2 \alpha + 4) y_1 x_2^2
            \\
            &\qquad + (1 -\alpha) y_1 x_2 - y_1
            + (\alpha + 1) x_2 y_2 + y_2 \rangle.
        \end{aligned}
    \end{equation}
    The second projection from the corresponding divisor to $X$ has degree $2$.
    Alternatively, the image of a point $P = (v, w)$ of $X$ under the morphism $X
    \to \Sym^2 (X)$ is described by the equation $x^2 + a_1 x + a_2 = 0$, $y =
    b_1 x + b_2$, where
    \begin{equation}
        \begin{aligned}
            a_1 =&
            \frac{-5 \alpha  v^2 + (\alpha +2) v}{ 5 v^2 -5 \alpha  v + (2 \alpha -1)},
            \\
            a_2 =& \frac{ (2 \alpha - 1) v^2}{ 5 v^2 - 5 \alpha v +  (2 \alpha - 1)},
            \\
            b_1 =&
            \frac{  -(7 \alpha + 4) v^2 w + (6 \alpha + 2) v w - 2 w}{5 v^5 + 5 (1 -2
                \alpha) v^4 + (3 -\alpha ) v^3 +  (7 \alpha - 1) v^2 -  (2 \alpha + 3) v
            + 1},
            \\
            b_2 =& \frac{ (3 \alpha + 1) v^2 w - (2 \alpha + 1) v w +  w}{5 v^5 + 5
                (1 -2 \alpha ) v^4 + (3 -\alpha ) v^3 + (7 \alpha - 1) v^2 - (2 \alpha +
            3) v + 1} .
        \end{aligned}
    \end{equation}
    The first of these calculations need $40$ terms in the Puiseux expansion,
    whereas the latter needs $172$. The combination of these calculations takes
    around $2.5$ CPU seconds.
\end{example}

\begin{example}
    As a second example, we consider the curve
    \href{http://www.lmfdb.org/Genus2Curve/Q/20736/l/373248/1}{\texttt{20736.l.373248.1}}
    with simplified Weierstrass model
    \begin{equation}
        X \colon y^2 = 24 x^5 + 36 x^4 - 4 x^3 - 12 x^2 + 1.
    \end{equation}
    We find that this curve has QM over $\Qbar$ by a non-Eichler order of reduced
    discriminant $36$ in the indefinite quaternion algebra over $\Q$ with
    discriminant $6$. The full ring of endomorphisms is only defined over
    $\Q(\theta)$ where $\theta$ is a root of $x^8 + 4 x^6 + 10 x^4 + 24 x^2 +
    36$. Over the smaller field $\Q (\sqrt{-3})$ we get the endomorphism ring $\Z
    [3 \sqrt{-1}]$. A generator $\alpha$ with $\alpha^2=-9$ has tangent
    representation
    \begin{equation}
        M = \begin{pmatrix}
            -\sqrt{-3} & 2 \sqrt{-3}
            \\
            \sqrt{-3} & \sqrt{-3}
        \end{pmatrix}.
    \end{equation}
    Our algorithms can perform the corresponding verification over the field $\Q
    (\sqrt{-3})$ itself, by using the base point $P_0 = (0, 1)$. The second
    projection from the corresponding divisor to $X$ has degree $18$, and using
    the Cantor representation one needs functions of degree up to $105$. The
    corresponding number of terms needed in the Puiseux expansion is $128$ in the
    former case and $346$ in the latter. This time the calculations take around
    $8.5$ CPU seconds to finish.
\end{example}

\begin{example}
    A third example in genus $2$ is
    \href{http://www.lmfdb.org/Genus2Curve/Q/294/a/8232/1}
    {\texttt{294.a.8232.1}} with model
    \begin{equation}
        X\colon y^2 = x^6 - 8x^4 + 2x^3 + 16x^2 - 36x - 55.
    \end{equation}

    The endomorphism ring of this curve is of index $2$ in the ring $\Z \times
    \Z$. The methods of \ref{sec:splitjac} show that it admits two maps of degree
    $2$ to the elliptic curves
    \begin{equation}
        E_1 \colon y^2 = x^3 + 3440/3 x - 677248/27
        \quad \text{and} \quad
        E_2 \colon y^2 = x^3 + x^3 + 752/3 x - 9088/27
    \end{equation}
    The maps send a point $(x, y)$ of $X$ to
    \begin{multline}
        \biggl(
            \frac{24 x^4 + 72 x^3 + 4 x^2 - 24 x y - 200 x - 72 y - 200}{3 (x+2)^2},
            \\
            \frac{32 x^6 + 144 x^5 + 16 x^4 - 32 x^3 y - 768 x^3 - 144 x^2 y - 656
            x^2 - 144 x y + 1488 x + 1792} {(x+2)^3}
        \biggr)
    \end{multline}
    on $E_1$ and
    \begin{multline}
        \biggl(
            \frac{24 x^4 + 24 x^3 - 92 x^2 - 24 x y - 56 x - 24 y + 88}{3 (x+2)^2},
            \\
            \frac{-32 x^6 - 48 x^5 + 176 x^4 + 32 x^3 y + 224 x^3 + 48 x^2 y - 304
            x^2 - 48 x y - 240 x - 64 y + 192}{(x+2)^3}
        \biggr)
    \end{multline}
    on $E_2$. Finding these projections only requires $39$ terms of the Puiseux
    series, and takes around $1$ second.
\end{example}

\subsection{Examples in higher genus}

\begin{example}
    The final hyperelliptic curve that we consider is the curve
    \begin{equation}
        X \colon y^2 = x^8 - 12 x^7 + 50 x^6 - 108 x^5 + 131 x^4 - 76 x^3 - 10 x^2
        + 44 x - 19
    \end{equation}
    of genus $3$.  This is a model for the modular curve $X_0(35)$ over $\Q$,
    and in fact this equation was obtained as a modular equation satisfied by
    modular forms of level $35$. We could make some guesses about the
    endomorphism ring of its Jacobian by computing the space of cusp forms of
    weight $2$ and level $35$, but let us apply our algorithms as if we were
    ignorant of its modular provenance.

    We find that the Jacobian of $X$ splits into an elliptic curve and the
    Jacobian of a genus $2$ curve. Its endomorphism algebra $\Q \times \Q
    (\sqrt{17})$ is generated by an endomorphism whose tangent representation
    with respect to the standard basis of differentials $\left\{ x^i \d x / y
    \right\}_{i=1,2,3}$ is given by
    \begin{equation}
        \begin{pmatrix}
            1 & 0 & -1
            \\
            1 & -2 & 0
            \\
            -2 & -2 & 1
        \end{pmatrix} .
    \end{equation}
    which has characteristic polynomial $(t + 1) (t^2 - t - 4)$. The curve $X$
    admits a degree $2$ morphism to the elliptic curve $Y \colon x^3 + 6656/3 x
    - 185344/27$ which is given by
    \begin{equation}
        (x, y) \longmapsto
        \left(\frac{64 x^2 - 400 x + 272}{3 (x^2 - x - 1)},
        \frac{224 y}{\left(x^2-x-1\right)^2}\right) .
    \end{equation}
    Determining this projection again takes about a second. A curve that
    corresponds to the complementary factor dimension $2$ can be found by using
    the results by Ritzenthaler--Romagny in \cite{ritzenthaler-romagny}.
\end{example}

\begin{example}
    Our algorithms can equally well deal with more general curves. For example,
    it is known from work of Liang \cite{liang-thesis} that the plane quartic
    \begin{multline}
        X \colon x_0^4 + 8 x_0^3 x_2 + 2 x_0^2 x_1 x_2 + 25 x_0^2 x_2^2 - x_0 x_1^3 +
        2 x_0 x_1^2 x_2 + 8 x_0 x_1 x_2^2 + \\ 36 x_0 x_2^3 + x_1^4 - 2 x_1^3 x_2
        + 5 x_1^2 x_2^2 + 9  x_1 x_2^3 + 20 x_2^4 = 0
    \end{multline}
    has real multiplication by the algebra $\Q (\alpha)$, with $\alpha =
    2\cos(2\pi/7)$. We have independently verified this result. The equations
    for the divisor are too large to reproduce here, but they can be generated
    with the package \cite{cms-package}. The tangent representation of the
    endomorphism with respect to an echelonized basis of differential forms at
    the base point $P_0 = (-2 : 0 : 1)$ is of the rather pleasing form
    \begin{equation}
        \begin{pmatrix}
            \alpha & 0 & 0 \\
            0 & \alpha^2 - 2 & 0 \\
            0 & 0 & -\alpha^2 + \alpha + 1
        \end{pmatrix} .
    \end{equation}
    This verification takes about $7$ CPU seconds and requires Puiseux series
    with $66$ coefficients of precision.
\end{example}

\begin{example}
    As a final aside, we consider Picard curves of the form
    \begin{equation}
        X \colon y^3 = a_0 x^4 + a_2 x^2 + a_4.
    \end{equation}
    Petkova--Shiga \cite{petkova-shiga} have shown that the connected component
    of the  Sato--Tate group of a general such curve is equal to $\U (1) \times
    \SU (2)_2$. The endomorphism ring of such a general curve $X$ is of index $4$
    in a maximal order of $\Q (\zeta_3) \times B$, where $B$ is the indefinite
    quaternion algebra of discriminant $6$.

    The Jacobian of $X$ therefore splits into an elliptic curve with CM and the
    Jacobian of a curve $Y$ of genus $2$ that has QM. Once again the curve $Y$
    can be identified explicitly using recent work of Ritzenthaler--Romagny
    \cite{ritzenthaler-romagny}. It turns out that the field of definition of the
    endomorphism ring of $X$ is the splitting field of the polynomial $t^6 - (2^4
    (a_4/a_0) (a_2^2 - 4 a_0 a_4))$. Using our algorithms, an explicit expression
    of the correspondence between $X$ and $Y$ can also be obtained.
\end{example}

\bibliographystyle{alpha}
\bibliography{biblio}

\end{document}